\documentclass[11pt]{article}
\usepackage[affil-it]{authblk}

\usepackage{geometry}
\usepackage{graphicx}
\usepackage{enumerate}
\usepackage{fullpage}
\usepackage{url}
\usepackage{dsfont}

\usepackage{natbib}
\setcitestyle{numbers,open={[},close={]},comma}
\usepackage{hypernat}

\usepackage{color}
\definecolor{clemson-orange}{RGB}{234,106,32}
\definecolor{broncos-orange}{RGB}{252,76,2}
\definecolor{cincinnati-red}{RGB}{190,0,0}
\definecolor{chicago-maroon}{RGB}{128,0,0}
\definecolor{northwestern-purple}{RGB}{82,0,99}
\definecolor{pink}{RGB}{255,105,180}
\definecolor{celtics}{RGB}{46,123,59}
\definecolor{leafs-blue}{RGB}{0,58,120}
\definecolor{pure-cyan}{RGB}{0,100,92}
\definecolor{lawngreen}{RGB}{0,250,154}

\newcommand{\bb}{\mathbb}
\newcommand{\ex}{{\bb E}}

\newcommand{\R}{\bb R}
\newcommand{\BR}{\bb R}
\newcommand{\E}{\bb E}

\newcommand{\bone}{\mathds{1}}

\newcommand{\rank}{\textnormal{rank}}
\newcommand{\cch}{\overline{\text{conv}}\,}

\usepackage[colorlinks,citecolor=northwestern-purple,urlcolor=chicago-maroon,linkcolor=chicago-maroon]{hyperref}

\usepackage{amssymb}
\usepackage{amsfonts}
\usepackage{amsmath}
\usepackage{amsthm}

\DeclareMathOperator\dist{dist}
\DeclareMathOperator\ext{ext}

\DeclareMathOperator\cl{cl}
\DeclareMathOperator\bd{bd}

\usepackage[nameinlink]{cleveref}
\crefname{assumption}{Assumption}{Assumptions}
\crefname{lemma}{Lemma}{Lemmas}
\crefname{theorem}{Theorem}{Theorems}
\crefname{corollary}{Corollary}{Corollaries}
\crefname{proposition}{Proposition}{Propositions}
\crefname{claim}{Claim}{Claims}
\crefname{subclaim}{Subclaim}{Subclaims}
\crefname{procedure}{Procedure}{Procedures}
\crefname{algorithm}{Algorithm}{Algorithms}
\crefname{figure}{Figure}{Figures}
\crefname{remark}{Remark}{Remarks}
\crefname{definition}{Definition}{Definitions}
\crefname{table}{Table}{Tables}

\theoremstyle{definition}
\newtheorem{theorem}{Theorem}[section]

\newtheorem{claim}{Claim}
\newtheorem{subclaim}{Subclaim}

\newtheorem{definition}[theorem]{Definition}

\newtheorem{lemma}[theorem]{Lemma}

\newtheorem{proposition}[theorem]{Proposition}

\newtheorem{assumption}{Assumption}

\title{The discrete moment problem with nonconvex shape constraints}

\author{Xi Chen\thanks{Stern School of Business, New York University, E-mail: xchen3@stern.nyu.edu} \qquad
Simai He\thanks{School of Information Management and Engineering,
Shanghai University of Finance and Economics,
E-mail: simaihe@mail.shufe.edu.cn}  \qquad
Bo Jiang\thanks{School of Information Management and Engineering,
Shanghai University of Finance and Economics,
E-mail: jiangbo@mail.shufe.edu.cn} \qquad
Christopher Thomas Ryan\thanks{%
Booth School of Business, University of Chicago,
E-mail: chris.ryan@chicagobooth.edu} \qquad
Teng Zhang\thanks{Management Science and Engineering, Stanford University,
E-mail: tengz@stanford.edu}
}
\affil{}

\begin{document}

\maketitle

\begin{abstract}
The discrete moment problem is a foundational problem in distribution-free robust optimization, where the goal is to find a worst-case distribution that satisfies a given set of moments.
This paper studies the discrete moment problems with additional ``shape constraints'' that guarantee the worst case distribution is either log-concave or has an increasing failure rate. These classes of shape constraints have not previously been studied in the literature, in part due to their inherent nonconvexities. Nonetheless, these classes of distributions are useful in practice.
We characterize the structure of optimal extreme point distributions by developing new results in reverse convex optimization, a lesser-known tool previously employed in designing global optimization algorithms.
We are able to show, for example, that an optimal extreme point solution to a moment problem with $m$ moments and log-concave shape constraints is piecewise geometric with at most $m$ pieces. 
Moreover, this structure allows us to design an exact algorithm for computing optimal solutions in a low-dimensional space of parameters. Moreover,  We describe a computational approach to solving these low-dimensional problems, including numerical results for a representative set of instances.

\vskip 10pt

\end{abstract}

\section{Introduction}\label{s:introduction}

The moment problem is a classical problem in analysis and optimization, with
roots dating back to the middle of the nineteenth century. At that time, the
goal there was to seek to
bound tail probabilities and expectations with given distributional moment information. Pursuing this initial goal remains active to the present day. For example, \citet{Bertsimas2005} provides tight closed form bounds of $P(X\ge (1+\delta)EX)$ with given first three moments of a random variable $X$. \citet{He2010} extends the problem for random variables given first, second and forth order moments, which also provided the first
nontrivial bound for $P(X\ge EX)$.

Beyond these foundational questions, the moment problem serves as an important
building block in a variety of applications in the stochastic and robust optimization
literatures \citep{prekopa2013stochastic,popescu2005semidefinite,saghafian2016newsvendor,rujeerapaiboon2016chebyshev,tian2017moment}. In particular, moment
problem are foundational to distribution-free robust optimization, where insight into the
structure of optimal measures can
be used to devise algorithms and describe properties of optimal decisions. A
classic example of this approach is due to \citet{scarf1958min} who leverages
the fact that an optimal solution to the moment problem given the first two
moments is a sum of two Dirac measures. This insight provides an analytical
formula for the optimal inventory decision in a robust version of the
newsvendor problem. There is a vast literature on robust optimization that builds on these initial insights in a variety of facets (see, for instance, \cite{Goh2010,delage2010distributionally,jiang2012robust,bandi2012tractable,bertsimas2013data,long2014distributionally,gao2016distributionally,
chen2016distributionally,Natarajan2017} among many others).

The focus of this paper is the discrete moment problem,
an important special case of the general moment that is less
well-studied in the literature. In the discrete moment problem, the
underlying sample space is a discrete set. The work of Pr\'ekopa (see for
instance \citet{prekopa1990sharp}) made a fundamental
contribution by devising efficient linear programming methods to study
discrete moment problems. These approaches remain state-of-the-art and has
seen application in numerous areas including project management
\citep{prekopa2004new} and network reliability \citep{prekopa1991existence}. Project management has also been studied in the robust optimization (see, for instance, \cite{bertsimas2006persistence}).

In classical versions of the moment problem (including the works by Pr\'ekopa and his co-authors just cited), the only constraints arise from
specifying a finite number of moments. One criticism of this approach is that it can result in
bounds and conclusions that may be too weak to be
meaningful, or in the case of robust optimization with only moment
constraints, result in decisions that are too conservative. For instance,
Scarf's solution for the newsvendor problem may even suggest to not order
\emph{any} inventory even when the profit margin is high \citep{perakis2008regret}.
This has driven researchers to introduce additional constraints, including those
on the \emph{shape} of the distribution. For example, \citet{perakis2008regret} study the newsvendor problem leveraging non-moment information, including symmetry and unimodality.
\citet{Han2014} study the newvendor problem relaxing the usual assumption of risk neutrality.
\citet{saghafian2016newsvendor} analyze the problem with the bound of
tail probability and \citet{Natarajan2017} recently developed closed-form
solutions under asymmetric demand information. In all cases, more intuitive and less conservative inventory decisions result, when compared to the classical setting with moment information alone. Other robust optimization papers that consider shape constraints include
\citet{LiJiang:16} who study the chance constraints and conditional Value-at-Risk constraints when the distributional information consists of the first two moments and unimodality,  \citet{lam2015tail} who study tail distributions with convex-shape constraints, and \citet{Hanasuanto:15} who study the multi-item newsvendor problems with multimodal demand distributions.

However, introducing shape constraints brings new theoretical challenges. A seminal paper by \citet{popescu2005semidefinite} provides a general framework for studying continuous moment problems under shape constraints that includes, among others, symmetry and unimodality. These moment problems are formulated as semi-definite programs (SDPs) that are polynomial time solvable. \citet{perakis2008regret} also employ Popescu's framework to provide analytical robust solutions to the newsvendor problem under shape constraints that are better behaved than classical Scarf solutions.  For the discrete moment problem, we are aware of only one paper \citep{subasi2009discrete} that considers shape constraints. \citet{subasi2009discrete} adapt Pr\'ekopa's linear programming (LP) methodology to include unimodality, which is modeled by an additional set of linear constraints.

Both \citet{popescu2005semidefinite} and \citet{subasi2009discrete} illustrate how a certain class of constraints can be adapted into existing computational frameworks, SDP-based in the case of \citeauthor{popescu2005semidefinite} and LP-based in the case of \citeauthor{subasi2009discrete}. However, there remains relevant shape constraints that are practical significant and do not naturally fit into these settings. In this paper we focus on two shape constraints: log-concavity (LC) and the increasing failure rate (IFR) property of discrete distributions (these are defined in \Cref{s:problem} below). Here, we briefly highlight the importance and the applications for each class of distributions.
\begin{enumerate}[(i)]
\item  LC measures arise naturally in many applications. For example, \citet{subasi2009discrete} illustrate how the length of a critical path in a PERT model where individual task times are described by beta distributions has a LC distribution but its other properties (other than moments inferred by the beta distributions) are unknown. Log-concavity has a wide range of applications to statistical modeling and estimation \citep{Walther:00}, e.g., \citet{Duembge:11} show how the log-concavity allows the estimation of a distribution based on arbitrarily censored data (which is a common form of data  for demand observations). The log-concavity also plays a critical role in economics \citep{Bagnoli:05}. For example, in contract theory, one commonly assumes that an agent's type is a LC random variable \citep{Laffont:88}.  The log-concavity of a distribution function has also been widely used in theory of regulation \citep{BaronMyerson:82,Lewis:88}, and in characterizing efficient auctions \citep{Myerson:83,Matthews:87}.

\item IFR distributions also play an important role in numerous applications in fields as wide-reaching as reliability theory \citep{barlow1996mathematical}, inventory management \citep{gavirneni1999value}, revenue management \citep{lariviere2006note} and contract theory \citep{dai2016impact,lariviere2001selling}. One reason for the prevalence IFR distributions in applications is that IFR distributions are closed under sums of random variables (and the associated convolutions of distribution functions). This is not the case for the shape properties studied by others, including symmetry and unimodality. The IFR property is useful in applications for simplifying optimality conditions to facilitate the derivation of properties of optimal decisions that yield managerial insights.
\end{enumerate}

In \Cref{s:problem} we show that the standard characterizations of discrete LC and IFR distributions, when added to the moment problem, make the resulting problem nonconvex and thus not amenable to either an SDP or LP formulation. Indeed, when \citet{subasi2009discrete} derive LC distributions in their applications, they \emph{relax} the LC property to unimodality, a shape constraint that \emph{can} be approached by LP techniques.

At this point, one could turn to approximation methods, including conic-optimization techniques to solve the resulting nonconvex formulation. It is well known that the copositive cone and its dual are powerful tools to could convert nonconvex problems equivalently into convex ones (see, e.g., \citep{pena2015completely, xu2016copositive, burer2009copositive, hanasusanto2016conic}). For instance, the LC discrete moments problem considered here can be cast as a completely positive conic problem \citep{pena2015completely}. Despite this convexity, the resulting problem is still computationally intractable and further relaxation is required to obtain an approximate solution.

We do not follow an approximation approach. The nonconvexities that arise in our problems are of a certain type that can be leveraged to provide an exact global optimization algorithm and analytical results on the structure of optimal solutions. Indeed, the feasible regions have \emph{reverse convex} properties (as introduced in \citep{meyer1970validity} and later developed in \citep{hillestad1980reverse} among others). A set is reverse convex if its complement is convex. Reverse convex programming is a little-studied field that has largely found application in the global optimization literature (see, for instance, \citet{horst1999dc}). To our knowledge, this theory has not been applied in the robust optimization literature.

In \Cref{s:optimization-theory} we extend standard results in the reverse convex programming literature (in particular, those of \cite{hillestad1980reverse}) so that they are applicable to our setting by introducing the notion of reverse convexity \emph{relative} to another set. The main benefit is that we can show reverse convex programs of this type have the following appealing structure --- there exist optimal extreme point solutions with a basic feasible structure analogous to basic feasible solutions in linear programming. The basic feasible structure reveals (in \Cref{s:characterize}) that optimal extreme point distributions in the LC and IFR settings have piecewise geometric structure.  This analytical characterization allows for solving these moment problems as low-dimensional systems of polynomials equations. We propose a specialized computation scheme for working with such systems. This allows us to provide numerical bounds on probabilities that are tighter than those in the existing literature, including those bounds that leverage unimodal shape constraints (see \Cref{s:computation-and-numerical}). All proofs not in the main text are included in the appendix.

\subsection*{Summary of contributions}

The main focus of the paper is on theoretical properties of LC and IFR-constrained moment problems, where we provide structural results on optimal solutions. For the LC case we show there exists optimal solutions that are piecewise geometric, and for the IFR case we show the tail probabilities of optimal distributions are piecewise geometric.

Our structural results and computational approaches suggest a wide range of applications due the prevalence of these classes of shape-constraints in real applications, as discussed above. Our results can provide new bounds on tail inequalities (i.e., $\Pr(X \geq a)$) for a random variable $X$ under moment and shape constraints. We provide a numerical framework for computing these bounds.

Moreover, the techniques developed in this paper allow us to solve an inner maximization problem with LC and IFR constraints. Our structural results could prove useful in solving the outer minimization problem in a robust optimization framework. Indeed, solution approaches to the standard max-min robust optimization formulation benefit greatly when the inner maximization problem has analytical structure.

Finally, we prove a new result on a generalized form of reverse convex optimization (\cref{Thm:general-nonconvex}) that may be of independent interest, with potential applications to other nonconvex optimization problems.

\subsection*{Notations}

We use the following notation throughout the paper. Let $\R$ denote the set of real numbers and $\R^n$ the vector space of $n$-dimensional real vectors. Moreover, let $\R^n_+$ denote the set of $n$-dimensional vectors with all nonnegative components and $\R^n_{++}$ denote the set of $n$-dimensional vectors will all positive components. The closure of the set $S$ in $\R^n$ (in the usual topology) is denoted $\cl(S)$ and its boundary by $\bd(S)$. Let $\E[\cdot]$ denote the expectation operator and $\bone_A$ the indicator function of set $A$.

Let $[k, \ell] = \left\{k, k+1, \dots, \ell-1, \ell \right\}$ denote the set of consecutive integers, starting with integer $k$ and ending with integer $\ell$. Similarly, let $(k,\ell) = \left\{k+1,k+2, \dots, \ell-2, \ell-1\right\}$. We will not have occasion to use $[\cdot,\cdot]$ and $(\cdot,\cdot)$ in their usual sense as intervals in $\R$, so there is no chance for confusion. For $k$, $j$ positive integers, $\binom{k}{j}$ denotes the binomial coefficient of $k$ choose $j$; that is, it counts the number of ways to choose $j$-subsets of $k$ objects.

\section{The discrete moment problem with nonconvex shape constraints}\label{s:problem}

We study the classical problem of moments with $m$ moments (cf. \cite{popescu2005semidefinite}):
\begin{align*}
\max_{\mu \in \mathcal P} \  & \int_\Omega f(w) d\mu \\
\text{s.t. } & \int_\Omega w^i d\mu(w) = q_i  \text{ for } i \in [0, m] \nonumber
\end{align*}
where $\mathcal P$ is a subset of measures $\mu$ on the measurable space $\Omega$ (with elements denoted by $w$) with $\sigma$-algebra $\mathcal B$, $f$ is a measurable function and $q_i \in \R$ for $i \in [0, m]$. We take $q_0 = 1$ to ensure that $\mu$ is a probability measure and the remaining $m$ constraints correspond to requiring the measure $\mu$ has $q_1, \dots, q_m$ as its first $m$ moments.

Our focus is where $\Omega = \left\{w_1, \dots, w_n\right\} \subseteq \R$ is
a finite set of real numbers and $\mathcal B$ is the power set of $\Omega$.
In fact, we assume that $\Omega = \left\{1,2,\dots,n\right\}$ and so $w_j =
j$ (however, see Section~\ref{ss:computational-method} where we have occasion
to rescale the $w_j$). In this setting, a measure $\mu$ can be represented by
a nonnegative $n$-dimensional vector $(x_1, \dots, x_n)$ where $\mu({w_j}) =
x_j$ and $f(w_j) = f_j$ for $j \in [1,n]$. We will refer to the vector $(x_1,
\dots, x_n)$ as a \emph{distribution} and often suppress the measure $\mu_x$
that it represents. This yields the following \emph{discrete moment problem}
(DMP):

\begin{subequations}\label{eq:discrete-moment-problem}
\begin{align}
&&\max_{x \in \R^n} \  & \sum_{j=1}^n f_j x_j  \label{eq:dmp-objective}
\qquad\qquad\qquad\qquad\qquad\qquad\qquad\qquad\qquad\\
(\text{DMP}) && \text{s.t. } & \sum_{j=1}^n w_j^i x_j = q_i \text{ for } i \in [0, m] \label{eq:dmp-moments} \\
&& & \mu_x \in \mathcal P. \label{eq:dmp-shape-constraints}
\end{align}
\end{subequations}
We study \eqref{eq:discrete-moment-problem} for two specifications of the set of distributions $\mathcal P$ in constraint \eqref{eq:dmp-shape-constraints}.

\begin{definition}[cf. Definition~2.2 in \cite{canonne2015testing}]\label{definition:log-concave}
A distribution $x = (x_1, \dots, x_n)$ is \emph{discrete log-concave} (or simply \emph{log-concave} or (LC)) if (i) for any $1 \le k < j < \ell \le n$ such that $x_kx_\ell > 0$ then $x_j > 0$; and (ii) for all $j \in (1,n)$, $x_{j-1}x_{j+1} \le x_j^2$. We let $\mathcal P_{\text{LC}}$ denote the class of all LC distributions.
\end{definition}

More precisely, (i) implies that for every LC distribution there exists a \emph{consecutive support} $[k, \ell]$ for some $1 \le k \le \ell \le n$ such that $x_j > 0$ for $j \in [k,\ell]$ and $x_j = 0$ otherwise. For an LC distribution $x$ with support $[k, \ell]$ we must then ensure $x_{j-1}x_{j+1} \le x_j^2$ holds for $j \in (k, \ell)$. At all other $j$ the inequality is trivial because at least one of $x_{j-1}$ or $x_{j+1}$ is zero.

\begin{definition}[cf. Definition~2.4 in \cite{canonne2015testing}]\label{definition:ifr}
A distribution $x = (x_1, \dots, x_n)$ has an \emph{increasing failure rate} (IFR) if the failure rate sequence $r_j := \tfrac{x_j}{\sum_{k=j}^n x_k}$ is a non-decreasing sequence; that is, $r_k \ge r_j$ for all $k \ge j$. We let $\mathcal P_{\text{IFR}}$ denote the class of all IFR distributions.
\end{definition}

It is well-known that $\mathcal P_{\text{LC}}$ is a strict subset of $\mathcal P_{\text{IFR}}$ \citep{an1997log}.
%
It is relatively straightforward to see that the sets $\mathcal P_{\text{LC}}$ and $\mathcal P_{\text{IFR}}$ are nonconvex. However, they share one additional common feature that is critical to our approach.

\begin{definition}\label{definition:reverse-convex}
A set $R$ in $\R^n$ is \emph{reverse convex} if $R = \R^n \setminus S$ for some convex set $S \subseteq \R^n$.  A set $R$ is said to be \emph{reverse convex with respect to (w.r.t)  a set} $T \subseteq \R^n$ if $R = T \setminus S$ for some convex set $S$.
\end{definition}

In the remainder of this section we show that the problem (DMP) when setting $\mathcal P $ be $\mathcal P_{\text{LC}}$ or $\mathcal P_{\text{IFR}}$ all have constraints that are reverse convex w.r.t. $\R^n_+$. This common fact is leveraged to solve these related problems to global optimality in a unified framework.

The seemingly more or less straightforward generalization to reverse convexity w.r.t. $\R^n_+$, however, could lead to a significantly different analytical properties. For example, observe that if a function $f : \R^n \to \R$ is quasiconcave (over $\R^n$) then its lower level sets are reverse convex. However, a function whose lower level sets are reverse convex w.r.t. some strict subset $T$ of $\R^n$ need not be quasiconcave.
In \Cref{s:optimization-theory} we show that problems with reverse convex structure can be approached using a novel optimization technique that extends the pioneering work of \cite{hillestad1980reverse}. 


\subsection{The moment problem over log-concave distributions}\label{ss:log-concave-problem}

Consider problem (DMP) when $\mathcal P = \mathcal P_{\text{LC}}$. We separate the optimization over $x$ into first determining a support (mapping to condition (i) of \cref{definition:log-concave}) and then introducing inequalities of the form $x_{j-1}x_{j+1} \le x_j^2$ for $j$ in that support (mapping to condition (ii) in \cref{definition:log-concave}). This yields the two-stage optimization problem:
\begin{subequations}\label{eq:DMP-LC}
\begin{align}
&&\max_{k, \ell: 1\le k\le \ell \le n} \ \max_{x \in \R^n} \  & \sum_{j=k}^\ell f_j x_j
\qquad\qquad\qquad\qquad\qquad\qquad\qquad\qquad\qquad\\
(\text{DMP-LC}) && \text{s.t. } & \sum_{j=k}^\ell w_j^i x_j = q_i \text{ for } i \in [0, m] \label{eq:dmp-lc-moments}\\
            && & x_{j-1}x_{j+1} \le x_j^2 \text{\ \ for } j \in (k, \ell) \label{eq:dmp-bad-constraints} \\
            && & x_j > 0 \text{ for } j \in [k, \ell] \label{eq:dmp-strict} \\
            && & x_j = 0 \text{ for } j \notin [k, \ell]. \label{eq:zero-outside-support}
\end{align}
\end{subequations}
The strict constraints \eqref{eq:dmp-strict} make the feasible region appear not to be closed. However, the following reformulation of (DMP-LC) reveals that the feasible region can be described with non-strict inequalities and is thus closed:
\begin{subequations}\label{eq:DMP-LCPRIME}
\begin{align}
&&\max_{x \in \R^n} \  & \sum_{j=1}^n f_j x_j
\qquad\qquad\qquad\qquad\qquad\qquad\qquad\qquad\qquad\\
(\text{DMP-LC'}) &&\text{s.t. } & \sum_{j=1}^n w_j^i x_j = q_i \text{ for } i \in [0, m]
\label{eq:dmp-lc-prime-moments} \\
            && & x_{j-u}^v x_{j+v}^u \le x_j^{u+v} \text{\ \ for } j \in (1,n), u \in [1,j-1] , v \in [1,n-j] \label{eq:dmp-good-constraints} \\
            &&  & x_j \ge 0 \text{ for } j \in [1, n]. \label{eq:dmp-lc-prime-nonnegativity}
\end{align}
\end{subequations}
In (DMP-LC') there is no outer maximization over the support between $k$ and $\ell$.

\begin{proposition}\label{prop:equivalent-formulations-lc}
Problems (DMP-LC) and (DMP-LC') are equivalent.
\end{proposition}

\cref{prop:set-convex-LC} below shows that (DMP-LC') is a nonconvex optimization problem where constraint \eqref{eq:dmp-good-constraints} defines a reverse convex set w.r.t. $\R^n_+$.

\begin{proposition}\label{prop:set-convex-LC} The set $\{ (x,y,z) : x^{u} y^{v} > z^{u + v},\; x \ge 0, y \ge 0, z \ge 0 \}$ is convex for any positive integers $u$ and $v$.
\end{proposition}

Whereas the set $\{ (x,y,z) : x^{u} y^{v} > z^{u + v},\; x \ge 0, y \ge 0, z \ge 0 \}$ is convex, the set where nonnegativity is relaxed -- that is, $S = \{ (x,y,z) : x^{u} y^{v} > z^{u + v} \}$ -- is \emph{not} convex. Indeed, $(-2,-1,0)$ and $(1,2,0)$ are in $S$ but $1/2(-2,-1,0) + 1/2(1,2,0) = (-1/2, 1/2,0)$ is not in $S$. This means that $f(x,y,z) = x^u y^v - z^{u + v}$ is not quasiconcave on its domain.

\subsection{The moment problem over increasing failure rate distributions}\label{ss:ifr-problem}

Consider problem (DMP) with $\mathcal P = \mathcal P_{\text{IFR}}$. The following result illustrates a tight connection between the IFR case and the LC case. This result is known in the continuous case (see \cite[Chapter~2]{barlow1996mathematical}), we provide details for the discrete analogue that is the focus of this paper.


\begin{lemma}\label{lemma: IFR-tail-definition}
A distribution $x = (x_1,\dots,x_n)$ has an increasing failure rate if and only if its tail probability sequence $\{\bar F_1,
\dots, \bar F_n\}$ is log-concave, where $\bar F_j = \sum_{k=j}^n x_k$.
\end{lemma}

In the IFR case, (DMP) becomes
%
%
\begin{subequations}\label{eq:DMP-IFR}
\begin{align}
&&\max_{x \in \R^n} \  & \sum_{j=1}^n f_j x_j
\qquad\qquad\qquad\qquad\qquad\qquad\qquad\qquad\qquad \\
(\text{DMP-IFR}) && \text{s.t. } & \sum_{j=1}^n w_j^i x_j = q_i \text{ for } i \in [0, m] \\
            && & \tfrac{x_j}{\sum_{k=j}^n x_k} \text{ is non-decreasing in $i$.} \label{eq:monotonicity-non-tail}
\end{align}
\end{subequations}
Using the transformation described in \cref{lemma: IFR-tail-definition}, where $y_j = \sum_{k=j}^n x_k$ denotes tail probabilities, we can reformulate \eqref{eq:DMP-IFR} as
\begin{subequations}\label{eq:DMP-IFR-PRIME}
\begin{align}
&&\max_{y \in \R^n} \  & \sum_{j=1}^n f_j (y_j - y_{j+1})
\qquad\qquad\qquad\qquad\qquad\qquad\qquad\qquad\qquad \\
(\text{DMP-IFR'}) && \text{s.t. } & \sum_{j=1}^n (w_j^i-w_{j-1}^i) y_j = q_i \text{\ \ for } i \in [0, m]  \\
            && & y_{j-1}y_{j+1} \le y_j^2 \text{\ \ for } j \in (1, n) \label{eq:ifr-prime-logconcave}\\
            && & y_j - y_{j+1} \ge 0 \text{\ \ for } j \in [1, n) \label{eq:ifr-prime-monotonicity}\\
            && & y_j \ge 0 \text{\ \ for } j \in [1, n]
\end{align}
\end{subequations}
where $f_{0},\ w_{0}^j$ are set to $0$. Constraint \eqref{eq:ifr-prime-logconcave} captures the log-concavity of the tail probabilities and \eqref{eq:ifr-prime-monotonicity} captures the non-increasing property of tail probabilities. There is no need to consider an outer optimization over supports and use strict inequalities to capture the property of consecutive supports. The consecutiveness of supports is immediate from the monotonicity condition of the $y_j$. Indeed, once $y_j = 0$ for some $j$ then $y_k = 0$ for all $k > j$ by monotonicity.

\section{A special class of nonconvex optimization problems}\label{s:optimization-theory}

In this section we present a general class of problems that includes all the problems introduced in \Cref{s:problem} as special cases. This class admits optimal extreme point solutions that are determined by setting a sufficient number of inequalities to equalities. This result is reminiscent of linear programming where extreme points have algebraic characterizations as basic feasible solutions.

Our analysis proceeds in two stages. First, we discuss a broad class of optimization problems that have  optimal extreme point solutions. 
Second, we specialize this general class to a class of nonconvex optimization problems where the source of nonconvexity arises from reverse convex sets (see \cref{definition:reverse-convex}). This work extends some of theory on reverse convex optimization, initiated by \cite{hillestad1980reverse} but tailors these results to the discrete moment problem. To our knowledge, these results are not subsumed by others in the existing literature.

\subsection{Linear optimization over (nonconvex) compact sets}\label{ss:linear-optimization-over-compact-sets}

Let us first consider a very general optimization problem:
\begin{align}\label{eq:linear-optimization-problem}
\begin{split}
\min \ & c(x) \\
\text{ s.t. } & x \in S
\end{split}
\end{align}
where $c$ is a lower semicontinuous and quasiconcave function and $S$ is nonempty and compact (closed and bounded) subset of $\R^n$. It is worthwhile to note that the results in this section can be generalized to any locally convex topological vector space in the sense of \citet[Chapter~5]{hitchhiker}. This is not required for the study of the discrete moment problem, but is potentially relevant for an exploration of the continuous case that follows a similar line of inquiry.

The goal of this subsection is to prove the following:

\begin{theorem}\label{theorem:restrict-attention-to-extreme-points}
There exists an optimal solution to \eqref{eq:linear-optimization-problem} that is an extreme point of $S$.
\end{theorem}

Recall that an \emph{extreme point} of $S$ is any point $x \in S$ where the set of $d$ such that $x \pm \epsilon d \in S$ for some $\epsilon >0$ is empty. Let $\ext S$ denote the extreme points of the set $S$. The special case to \cref{theorem:restrict-attention-to-extreme-points} where $S$ is convex well-known and immediate from \citet[Corollary~7.75]{hitchhiker}:
\begin{lemma}\label{lemma:s-convex}
If $S$ is compact and convex then \eqref{eq:linear-optimization-problem} has an  optimal extreme point solution.
\end{lemma}
The proof when $S$ is not convex takes a couple more steps. The first step is to work with the closed convex hull $\cch S$ of $S$, which is the intersection of all closed convex sets that contain $S$.

\begin{lemma}(Theorem 5.3 in \cite{hitchhiker})\label{lemma:closed-convex-full-compact}
The closed convex hull of a compact set is compact. In particular, $\cch S$ is a compact convex set.
\end{lemma}

The following lemma helps us to leverage these results about closed convex hulls to learn about the original problem \eqref{eq:linear-optimization-problem}.

\begin{lemma}\label{lemma:extreme-points-in-set}
Let $S$ be a compact subset of $\R^n$. Then $\ext \cch S \subseteq \ext S$.
\end{lemma}

We prove \cref{theorem:restrict-attention-to-extreme-points} using \cref{lemma:closed-convex-full-compact,lemma:extreme-points-in-set} in Section \ref{sec:proof_restrict-attention-to-extreme-points} of the online supplement.

\subsection{Reverse convex optimization problem with nonnegative constraints}
\label{ss:application_to_noncave_minimization}

The following lemma captures the essence of reverse convex optimization and serves as motivation and a visualization tool for understanding our main theoretical result below (see \cref{Thm:general-nonconvex}). 

\begin{lemma}\label{lemma:boundary-conditions}
Consider the optimization problem
\begin{align*}
\min_{x \in \R^n} \ &  c(x) \\
s.t. \ & x \in R_p, \text{ for } p \in [1,P]
\end{align*}
where $c$ is a lower semicontinuous and quasiconcave function, $P \ge n$, and the $R_p$ are closed, reverse convex sets such that $X := \cap_p R_p$ is a nonempty and compact subset of $\R^n$. Then there exists an optimal solution that lies on the boundary of at least $n$ of the sets $R_p$.
\end{lemma}

\cref{lemma:boundary-conditions} extracts some ideas from existing results (particularly from \cite[Theorem~2]{hillestad1980reverse}) and presents them in a clean, geometric form.
To facilitate the understanding of this lemma, we further provide an intuitive graphical illustration in Figure \ref{fig:general_reverse}.
\begin{figure}[htbp]
  \centering
  \includegraphics[width=7cm]{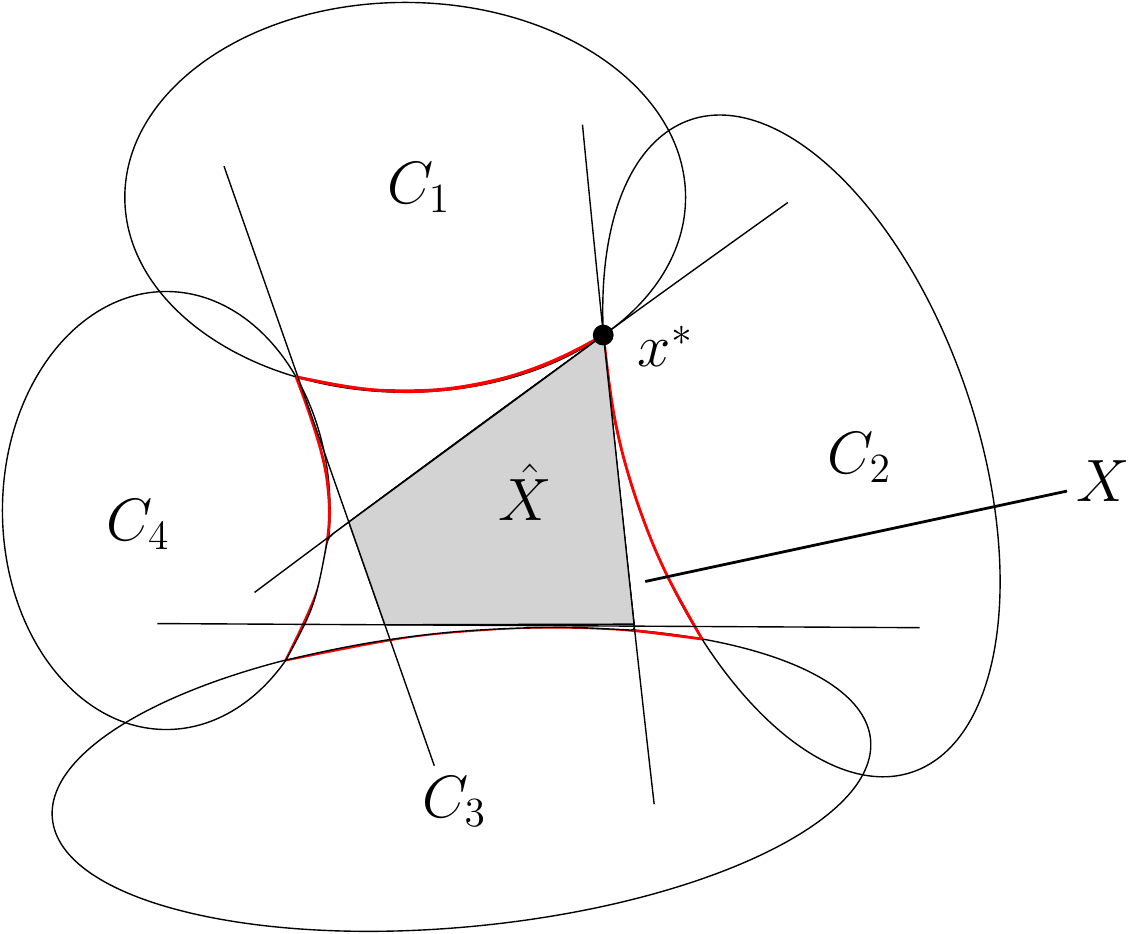}
  \caption{A illustration of the general reverse convex programming in $\mathbb R^2$. The feasible region $X$ is the intersection of several $R_p$, where each $R_p$ is the complement of a convex set $C_p$. After constructing the feasible polyhedron $\hat X$ (obtained via intersection of supporting hyperplanes of cl($C_p$) that weakly separate $x^*$),  we can show that the optimal extreme point solution $x^*$ is lies on the boundary of at least two of the sets $R_p$ using the theory of basic feasible solutions in linear programming.}
  \label{fig:general_reverse}
\end{figure}
Despite its elegance, this lemma is insufficient for our purposes. First, it
only applies when the $R_p$ are reverse convex. The argument breaks down if
the $R_p$ are reverse convex w.r.t. another convex set $S$, as
needed for the problems in \Cref{s:problem}.
In particular, when the convex set $S$ is a polytope, even though we can use
$R_p\cap S$ as a reverse convex set to replace $R_p$, it contains the
boundaries from the original polytope which are undesirable for analyzing the
extreme optimum solutions.
Second, the conclusion only provides a lower
bound on the number of boundaries an optimal solution lies on. Although
sufficient for the LC case, a strengthening that leverages the
concept of linear independence
--- familiar from the analogous linear programming result
\cite[Theorem~2.3]{bertsimas1997introduction} --- is needed for the IFR case.

As to the second insufficiency, 
a standard setting in reverse convex optimization is to consider a feasible region
\begin{equation*}
F = \left\{x \in \R^n : f_p(x) \le 0 \text{ for } p \in [1,P] \right\}
\end{equation*}
and assume properties on the functions $f_p$. These properties typically include differentiability assumptions (so that gradients are defined) and some form of concavity (the weakest being quasiconcavity). Under these concavity assumptions, the lower-level sets of $f_p$ are reverse convex and \cref{lemma:boundary-conditions} applies so that extreme points are determined by a minimum number of tight constraints of the form $f_p(x) = 0$. 
Unfortunately, those results do not apply in our setting. Indeed, the discrete moment problems we consider here does not involve quasiconcave functions, instead functions whose lower level sets are reverse convex \emph{w.r.t. the nonnegative orthant}.

These considerations motivate us to establish a more general theory of reverse convex optimization. In particular, we analyze the following problem
\begin{align*}\label{eq:reverse-convex-general}
\min \ & c(x)  \nonumber \\
\text{s.t.} \ & Ax = b  \nonumber \\
& f_p(x) \le 0 \text{ for } p \in [1,P] \tag{Rev-Cvx}\\
& x \ge 0, \nonumber
\end{align*}
where $c$ and the $f_p$ are functions from $\R^n$ to $\R$ and $A$ is an $m$ by $n$, and for $1 \leq p \leq P$, the set $\{x: f_p(x) \leq 0\}$ is reverse convex w.r.t. the nonnegative orthant $\R^n_+$.
\begin{assumption}\label{ass:assumptions-on-the-problem}
We make the following additional technical assumptions on \eqref{eq:reverse-convex-general}:
\begin{enumerate}[(i)]
\item The objective function $c(x)$ is continuous and quasiconcave,
\item The matrix $A$ is full-row rank,
\item For each $p$, $f_p$ is differentiable, and
\item The feasible region $X = \{ x \in \BR^n_+ : Ax = b,\; f_p(x) \le 0,\; p =1,\ldots,P \}$ is nonempty and compact.
\end{enumerate}
\end{assumption}
We also need the following notation to state the main theorem of this section. For any feasible solution $x$ to \eqref{definition:reverse-convex} let $S(x)$ denote the support of $x$; that is, $S(x) = \left\{j : x_j > 0\right\}$. Let $A^i$ denote the $i$-th row of the matrix $A$ and $A_j$ the $j$-th column.  For any subset $S$ of $[1,n]$ (for instance, the support of a feasible solution), let $A_{S} = [A_j]_{j \in S}$. That is, $A_S$ is the submatrix of $A$ consisting the columns indexed by $S$. Recall that $A^r_S$ denotes the $r$-th row of the matrix $A_S$ and let $\mathcal L(S) = \text{span} (\{(A_S^1)^\top, \dots, (A_S^m)^\top\}$ denote the span of the rows of $A_S$. Finally, let $\nabla f_p(x)$ denote the gradient of $f_p$ at $x$, where $[\nabla f_p(x)]_S$ is the gradient of $f_p$ restricted to the components in the subset $S$.

\begin{theorem}\label{Thm:general-nonconvex}
Consider an instance of \eqref{eq:reverse-convex-general} where \cref{ass:assumptions-on-the-problem} hold. Then there exists an optimal extreme point solution.

Moreover, for any extreme point optimal solution $x^*$, $n - m$ of the following $P+n$ inequalities
\begin{equation*}
\begin{array}{rll}
f_p(x^*) \le 0 & \text{ for }& p \in [1,P]\\
x^*_j \ge 0 & \text{ for } & j \in [1, n]
\end{array}
\end{equation*}
are tight.

 In addition, letting $S = S(x^*)$, if we further assume that for all the tight constraints $p$ with $f_p(x^*)  = 0$ one has $ [\nabla f_p(x^*)]_S \not\in \mathcal{L}(S) $, then there are $n -m$ of the vectors $\{ \nabla f_p(x^*) : f_p(x^*) = 0   \} \cup \{ e_j : x^*_j = 0\}$ are linearly independent, where $e_j$ is the unit vector with $1$ in the $j$th component and $0$ otherwise.
\end{theorem}

Theorem \ref{Thm:general-nonconvex} is the main theoretical result in this paper. 
The proof largely follows the geometric intuition captured in Figure~\ref{fig:general_reverse}. At its core, it involves defining  separating hyperplanes and inscribing a polyhedral set $\hat X$ inside the feasible region. Then, the equivalence of extreme points and basic feasible solutions for the polyhedron $\hat X$ is leveraged to establish the result.

However, the proof has additional technical challenges. It must make sense of
how inequalities that describe the orthant $\R^n_+$ interact with the
gradients of the constraint functions $f_p$.
Moreover, the affine equality constraints $Ax = b$, that correspond to the
moment conditions in \eqref{eq:discrete-moment-problem}, force us to work
within the affine space defined by these constraints for much of the proof.
Finally, we require a spanning condition of the gradients to ensure that the
full analysis can be captured in that space.

\begin{proof}[Proof of \cref{Thm:general-nonconvex}]

Since $c$ is continuous and quasiconcave and $X$ is a compact set, then by \cref{theorem:restrict-attention-to-extreme-points}, there exists an optimal extreme point solution. For any such optimal extreme point $x^*$ with support $S = S(x^*)$ define
\begin{equation*}
X_0 := \{ x : Ax = b,\;  x_j = 0 \text{ for }  j \not\in S, \; x_j > 0 \text{ for }  j \in S\}.
\end{equation*}
Then the feasible region $X$ includes $\{ x : f_p(x) \le 0 ,\; p =1,\ldots, P\} \cap X_0$.
Let $\delta_1 = \min\{x^*_j : x^*_j > 0 \}$ and denote
\begin{equation*}
X(\delta_1) := \{ x : Ax = b,\;  x_j = 0 \text{ for }  j \not\in S, \; x_j \ge \delta_1/ 2,  \text{ for }  j \in S\}.
\end{equation*}

Our goal is as follows. For $p=1,\ldots, P$, we want to construct sets $\hat{X}_p$ of the form
\begin{equation}\label{hat-X-p}
\hat{X}_p := \{ x : \alpha_p^{\top}(x - x^*)\le  \beta_p \} \cap X(\delta_1),
\end{equation}
such that
\begin{equation}\label{X-hat}
\hat{X} := \cap_{p=1}^{P} \hat{X}_p = \{ x : \alpha_p^{\top}(x - x^*)\le  \beta_p , \; p=1,\ldots, P\} \cap X(\delta_1)
\end{equation}
is a subset of $X$, where $\alpha_p$ and $\beta_p \ge 0$ will be specified later. As long as $\hat{X} \subseteq X$, since $x^*$ is an extreme point of $X$, it is an extreme point of $\hat{X}$ as well. Note that $\hat{X}$ is defined by a number of linear equalities and inequalities, then there must exists $n$ of them that are tight at point $x^*$, and we can further check which constraint is tight.

We now construct such a $\hat{X}$. Let $x_S = [x_j]_{j \in S}$ and
\begin{equation}\label{X-p}
X_p = X_0 \cap \{ (x_S ; 0 ) : f_p(x_S;0)>0 \}.
\end{equation}
A key property of $X_p$ is that it admits a strong separation property useful for our arguments (see \cref{complete-separation-Xp} below). To describe this property, we explore a related set in a smaller subspace. Construct matrix $B \in \BR^{|S| \times (|S| - \rank(A_S))}$ such that its columns  span the whole null space of $A_S$; That is $A_S B = 0$ and $\rank(B) = |S| - \rank(A_S)$. Then, we have that
\begin{equation}\label{transformation}
\{ (x_S;0) : A_S\, x_S = b  \} = \{  (By + x^*_S;0) : y \in  \BR^{(|S| - \rank(A_S))}\}.
\end{equation}
Letting
\begin{equation*}
Y_p := \{ y : By + x^*_S >0, f_p(By + x^*_S ; 0)>0 \},
\end{equation*}
we can define the ``strong separation'' property of $X_p$ as follows.

\begin{claim}(Strong separation)\label{complete-separation-Xp} For all $p$ there exists $\alpha_p^{\top}$ and $\hat \beta_p >0$ such that
\begin{equation}\label{general-separation}
\left\{ \begin{array}{rl}
\hat \alpha_p^{\top} (x - x^*) \ge \hat{\beta}_p > 0, \text{ for } x \in X_p & \text{ if } \quad 0 \not\in \cl(Y_p)\\
\hat \alpha_p^{\top} (x - x^*) > 0, \text{ for } x \in X_p & \text{ if } \quad 0 \in \cl(Y_p)
\end{array}\right.
\end{equation}
Moreover, if we further assume $[ \nabla f_p(x^*) ]_S \not\in \mathcal{L}(S) $ then $\nabla f(x^*)^{\top}(x - x^*)  > 0$ for all $x \in X_p$.%
%
%
%
%
\end{claim}
We relegate the proof of \cref{complete-separation-Xp} to Section \ref{sec:proof_claim-complete-seperation} in the supplement and return to constructing $\hat X$. According to \eqref{X-hat} it suffices to show how to construct {$\hat X_p$} such that
\begin{equation}\label{X-p-sebset-f-p}
{\hat{X}_p} \subseteq \{ x : f_p(x)\le  0 \}  ,\; p\in [1,P],
\end{equation}
since $X(\delta_1) \subseteq X_0$. In other words,
we need to prove that $x \in {\hat{X}_p}$ implies $f_p(x) \le 0$.

 \begin{table}
\centering
\begin{tabular}{|c|c|}
\hline
{$0 \notin \cl (Y_p)$} &  $\alpha_p = \hat{\alpha}_p$, $\beta_p = \hat{\beta}_p/2$, $\hat{\alpha}_p$ and $\hat{\beta}_p$ obtained by strong separation \\
\hline
{$0 \in \cl (Y_p)$} & $\alpha_p = \hat{\alpha}_p$, $\beta_p =0$, $\hat{\alpha}_p$ are obtained weak separation   \\
\hline
{$0 \in \cl (Y_p),\; [\nabla f_p(x^*)]_S \not\in \mathcal{L}(S)$ } & $\alpha_p = \nabla f_p(x^*)$ and $\beta_p =0$ \\
\hline
\end{tabular}
\caption{Specifying $\alpha_p$ and $\beta_p$ in {$\hat {X}_p$}.} \label{Table1}
\end{table}

We will show \eqref{X-p-sebset-f-p} in two cases: (i) $0 \not\in \cl(Y_p)$; (ii) $0 \in \cl(Y_p)$. We use \cref{Table1} to track some of the notation and details.

In case (i), according to \cref{complete-separation-Xp}, there exists some $\hat \alpha_p \neq 0$ and $\hat \beta_p \neq 0$ such that
\begin{equation}\label{strong-seperation}
\hat{\alpha}_p^{\top}(x - x^*)\ge  \hat{\beta}_p > 0\; \text{ for all }\; x \in { X_p}.
\end{equation}
By letting $\alpha_p = \hat{\alpha}_p$ and $\beta_p = \tfrac{\hat{\beta}_p}{2}$, one has $x^* \in {\hat X_p} \neq \emptyset$. Moreover, from definition \eqref{hat-X-p} of ${\hat X_p}$, any $x \in {\hat X_p}$ satisfies $x \in X(\delta_1) \subseteq X_0$ and
\begin{equation*}
\hat{\alpha}_p^{\top}(x - x^*) = {\alpha}_p^{\top}(x - x^*) \le  {\beta}_p = \hat{\beta}_p /2.
\end{equation*}
This combining with \eqref{strong-seperation} yields that $x \not\in {X_p}$ for any $x \in {\hat X_p}$. Then according to \eqref{X-p}, such $x$ does not belong to ${X_p}$ simply because it violates the constraint $f_p(x)>0$. Therefore we can conclude that $f_p(x) \le 0$ for all $x \in {\hat X_p}$.

In case (ii), again by \cref{complete-separation-Xp}, we have $\hat \alpha^{\top}(x - x^*)  > 0 \text{ for }  x \in X_p$, where $\hat \alpha = \nabla f(x^*)$ if $[ \nabla f_p(x^*) ]_S \not\in \mathcal{L}(S) $. Then we can take $\alpha_p = \nabla f_p(x^*)$ and $\beta_p = 0$ in \eqref{hat-X-p}. Obviously, $x^* \in \hat{X}_p \neq \emptyset$ and $x \not\in \hat{X}_p$ for any $x \in X_p$.
Similarly, we can argue that such a $x$ does not belong to $\hat X_p$ due to the violation of the constraint $f_p(x)>0$. Then it follows that $f_p(x) \le 0$ for all $x \in {\hat X_p}$.

So far, we have constructed ${\hat X_p}$ in the form of \eqref{X-p} as in Table 1 and $\hat{X}$ based on \eqref{X-hat}. Moreover, we have shown $\hat{X} \subseteq X$. Since $x^*$ is an extreme point of $X$ and lies both in $X$ and $\hat X$, it is an extreme point of $\hat{X}$ as well. Note that $\hat{X}$ is defined by a number of linear equalities and inequalities, then there must exists $n$ of them that are tight and linear independent at point $x^*$, by standard theory, e.g. \cite[Theorem~2.3]{bertsimas1997introduction}.

Since $A$ is an $m$ by $n$ matrix of rank $m$, there are $n-m$ tight constraints from
\begin{equation*}
\begin{array}{rl}
\alpha_p^{\top}(x - x^*)\le  \beta_p  & \text{ for } p \text{ such that } 0 \not\in \cl (Y_p) \\
{\alpha_p}^{\top}(x - x^*)\le  0 & \text{ for } p \text{ such that } 0 \in \cl(Y_p) \\
{x}_j = 0 & \mbox{for}\;  {j \not\in S }\\
{x}_j \ge \delta_1/2 & \mbox{for}\;  { j \in S},\\
\end{array}
\end{equation*}
where $\alpha_p = \nabla f_p$ if $[ \nabla f_p(x^*) ]_S \not\in \mathcal{L}(S) $. Now let's investigate which constraint in the above could be tight. First of all, it is obvious that $x^*_j = 0$ is tight for all {$j \not\in S$} and $x^*_j \ge \delta_1 >  \delta_1/2 $ could not be tight for all {$j \in S$}.
Then for the constraint $p$ such that $0 \not\in \cl(Y_p)$, since $\beta_p > 0$, $\alpha_p^{\top}(x^* - x^*) = 0 <  \beta_p$ cannot be tight.
Finally, recall we have proved in the previous discussion that $f_p({x^*})\ge 0$ for all $p$ such that {$0 \in \cl(Y_p)$}. That is, when $f_p(x^*)<0$, it holds that $0 \not\in \cl(Y_p)$ and thus the corresponding constraint $\alpha_p^{\top}(x - x^*)\le  \beta_p $ cannot be tight at $x^*$. In summary, all $n-m$ tight constraints come from
\begin{equation}\label{linear-tight}
\begin{array}{rl}
{\alpha_p}^{\top}(x - x^*)\le  0 & \mbox{for} \; p \; \mbox{such that}\; {f_p(x^*)=0 \; \mbox{and} \; 0 \in \cl(Y_p)}\\
{x}_j = 0 & \mbox{for}\;  {j \not\in S},\\
\end{array}
\end{equation}
which implies $n-m$ of the inequalities
\begin{equation*}
\begin{array}{rll}
f_p(x^*) \le 0 & \mbox{for}& p \in [1,P]\\
x^*_j \ge 0 & \mbox{for} & j \in [1,n]
\end{array}
\end{equation*}
in \eqref{eq:reverse-convex-general} are tight. Moreover, when $[ \nabla f_p(x^*) ]_S \not\in \mathcal{L}(S) $ for all $p$, $\alpha_p = \nabla f_p(x^*)$ in~\eqref{linear-tight} and these tight constraints are linearly independent. In other words, the set of vectors $\{ \nabla f_p(x^*): f_p(x^*) = 0   \} \cup \{ e_j: x^*_j = 0\}$ are linearly independent, where $e_j$ is the gradient of the constraint $x_j \ge 0$. This completes the proof of \cref{Thm:general-nonconvex}. 
\end{proof}


\section{Characterizing optimal extreme point solutions in the discrete moment problem}\label{s:characterize}

\cref{Thm:general-nonconvex,theorem:restrict-attention-to-extreme-points} are powerful tools for analyzing the moment problems we discussed in \Cref{s:problem}. They will allow us to characterize the structure of optimal extreme point solutions. In the following two subsections we analyze the LC and IFR distributions cases from \Cref{ss:log-concave-problem,ss:ifr-problem}. There is a general pattern to our analysis, which we briefly describe here.

Each problem has two alternate formulations, with one indicated by a ``prime''. In the LC case these two formulations are (DMP-LC) and (DMP-LC'). The ``prime'' formulation has a closed and compact feasible region which allows us to leverage \cref{theorem:restrict-attention-to-extreme-points} to show the existence of an optimal extreme point solution $x^*$. With $x^*$ in hand, we apply \cref{Thm:general-nonconvex} to a small adjustment of the ``non-prime''  formulation that replaces strict inequalities with non-strict inequalities based on the support of $x^*$. \cref{Thm:general-nonconvex} implies that a certain number of constraints are tight, including some number of the reverse convex constraints (for instance, \eqref{eq:dmp-bad-constraints} in (DMP-LC)). Making these constraints tight determines the structure of the optimal extreme point solutions. In the LC case, a piecewise geometric structure is obtained.

\subsection{Log-concavity}\label{ss:log-concave-solution}

Recall the two alternate formulations (DMP-LC) and (DMP-LC'). In particular, recall that there are $m+1$ moment constraints in \eqref{eq:dmp-lc-moments} and \eqref{eq:dmp-lc-prime-moments}.
\begin{theorem}\label{theorem:log-concave-solution}
Every feasible instance of (DMP-LC) has an optimal extreme point solution. Moreover, every  optimal extreme point solution $x^*$ has the following structure: there exist (i) integers $u_i$ and $v_i$ for $i \in [1, m]$ with $k = u_1 < v_1 = u_2 < v_2 \dots < v_{m-1} = u_m < v_m = \ell$ where $[k,\ell]$ is the support of $x^*$ and (ii) real parameters $\alpha_i > 0$, $0 < r_i < 1$ for $i \in [1,m]$ such that
\begin{equation}\label{geometric-form}
x^*_j =
\begin{cases}
\alpha_i r_i^{j - u_i} & \text{ for }  j \in [u_i, v_i]\\
0 & \text{ otherwise.}
\end{cases}
\end{equation}
That is, there exists an optimal solution to (DMP-LC) that has a piecewise geometric structure with (at most) $m$ pieces.
\end{theorem}
\begin{proof} Consider the (DMP-LC') representation of the problem. The zeroth order moment constraint (\eqref{eq:dmp-lc-prime-moments} for $i = 0$) is $\sum_{i=1}^{n} x_i =1$, which, along with the nonnegative constraints \eqref{eq:dmp-lc-prime-nonnegativity}, implies the feasible region of the problem (DMP-LC') is compact. Then by \cref{theorem:restrict-attention-to-extreme-points}, there exists an optimal extreme point solution to (DMP-LC') and thus also (DMP-LC) since these problems are equivalent (via \cref{prop:equivalent-formulations-lc}).

Let $x^*$ be any extreme optimal solution and for simplicity we assume its support is $[1, n]$ (the general case of suppose $[k,\ell]$ with $1 < k < \ell < n$ follows analogously). Note that when $n \le m$, there are at most $m$ points in the interval $[1,n]$, where each point  $x_j$, $j \in [1,n]$ could be viewed as a single piece and the conclusion readily follows. Therefore, in the remainder of the proof we assume $n \ge m+1$.

Let $\underline{x} := \min\{ x_j^* :  j \in [1,n] \}$ and define the following problem:
\begin{subequations}\label{log-concave-inner-restriction-problem}
\begin{align}
\max_{x \in \R^n} \  & \sum_{j=1}^n f_j x_j \\
\text{s.t. } & \sum_{j=1}^n w_j^i x_j = q_i \text{ for } i \in [0, m]  \\
             & x_{j-1}x_{j+1} \le x_j^2 \text{\ \ for } j \in (1,n -1) \label{eq:log-concave-inner-restriction-problem-shape}\\
             & x_j \ge \underline{x}/{2} \text{ for } j \in [1, n]. \label{eq:great-to-be-nonstrict}
\end{align}
\end{subequations}
Note that \eqref{log-concave-inner-restriction-problem} is a restriction of (DMP-LC) with a given support and replacing the strict inequalities in \eqref{eq:dmp-strict} with non-strict inequalities in \eqref{eq:great-to-be-nonstrict}. Note also that $x^*$ is an extreme optimal solution to (DMP-LC) and it is feasible to \eqref{log-concave-inner-restriction-problem}, hence $x^*$ is an extreme optimal solution to \eqref{log-concave-inner-restriction-problem}.

To uncover the structure \eqref{geometric-form} of $x^*$ we apply \cref{Thm:general-nonconvex}. Convert the constraint $x_j \ge \underline{x}/{2}$ as a nonnegative constraint to mimic the nonnegativity constraint of \eqref{definition:reverse-convex} by making a change of variables $y_j := x_j -  \underline{x}/2$ to arrive at the following equivalent form:
\begin{subequations}\label{log-concave-inner-restriction-equivalence}
\begin{align}
\max_{x \in \R^n} \  & \sum_{j = 1}^{n} f_j y_j + (\underline{x}/2)  \sum_{j = 1}^{n} f_j \\
\text{s.t. } & \sum_{j=1}^{n} w_j^i y_j = q_i - (\underline{x}/2)  \sum_{j=1}^{n}w_j^i  \text{ for }  i \in [0, m]  \\
             & y_{j-1}y_{j+1} +  (\underline{x}/2)  (y_{j-1} + y_{j+1})\le y_j^2 + \underline{x} \cdot y_j \text{ for } j \in (1,n) \label{eq:reverse-convex-reformulation}\\
             & y_j \ge 0,\text{ for } j \in [1, n].
\end{align}
\end{subequations}
Observe that $y^* := x^* - \underline{x}/2$ is an optimal extreme point solution of \eqref{log-concave-inner-restriction-equivalence}.

We now verify that \eqref{log-concave-inner-restriction-equivalence} satisfies the conditions of \cref{Thm:general-nonconvex}. Again, the zeroth order moment constraint  guarantees the feasible region is compact. Let $f_j(y) = y_{j-1}y_{j+1} +  (\underline{x}/2)  (y_{j-1} + y_{j+1}) - y_j^2 - \underline{x} \cdot y_j$ for $j \in (1,n)$. Here the index $j$ plays the role of index $p$ in \cref{Thm:general-nonconvex}. Note that $p$ (the index of the constraint functions) need not be tied to $j$ (the index of the decision variable components) in a general application of \cref{Thm:general-nonconvex}. As we have shown in \cref{prop:set-convex-LC}, the set $\{x : x_{j-1} x_{j+1} > x_j^2, x \ge 0\}$ is convex, and it is an easy extension that $\{x : x_{j-1} x_{j+1} > x_j^2, x \ge \underline{x} > 0\}$ and this implies that $\{y : f_j(y) >0, y \ge 0\}$ is convex. This implies that all of the conditions in \cref{Thm:general-nonconvex} are satisfied when applied to \eqref{log-concave-inner-restriction-equivalence}.

Since the constraints $y_j \ge 0$ cannot be tight at point $y^*$ for $j \in [1,n]$, this application of \cref{Thm:general-nonconvex} implies that at least $n - m - 1$ of the \eqref{eq:reverse-convex-reformulation} constraints are tight at $y^*$, or equivalently there are at most $m - 1$ of the \eqref{eq:log-concave-inner-restriction-problem-shape} constraints that are not tight at $x^*$ in \eqref{eq:log-concave-inner-restriction-problem-shape}. These non-tight indexes can divide the interval $[1, n]$ into at most $m$ pieces, and within each piece we have $x_{j-1}x_{j+1}  = x_{j}^2, \text{ for } j \in [u_i, v_i]$, where $u_i$, $v_i$ are the left and right endpoint of piece $i$ of the domain. It is a standard observation to note that such a system implies $x_j = r_i^{j - u_i} x_{u_i}$ for $j \in [u_i,v_i]$. Setting $\alpha_i = x_{u_i}$ yields the form \eqref{geometric-form}. 
\end{proof}

The piecewise geometric form \eqref{geometric-form} of optimal extreme point distributions to (DMP-LC) is illustrated in \cref{fig:piecewise-geometric-form}.
\begin{figure}[!t]
\centering
\includegraphics[width=0.5\textwidth]{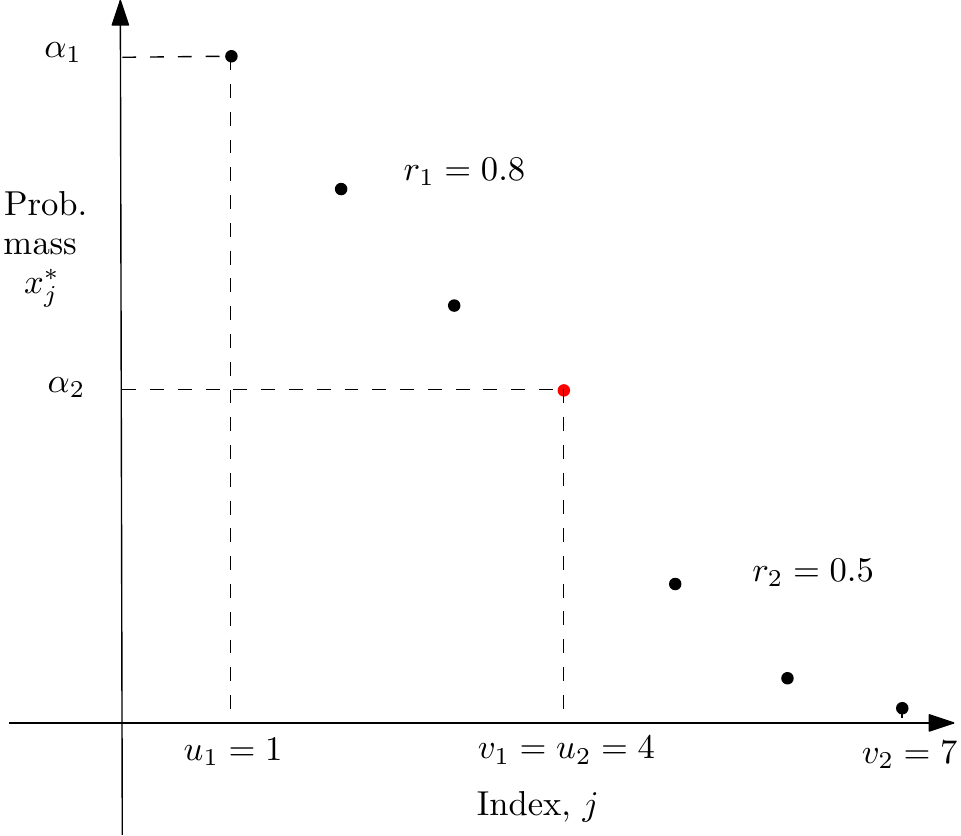}
\caption{Piecewise geometric structure of optimal extreme point solutions for a problem with $m = 2$.}
\label{fig:piecewise-geometric-form}
\end{figure}

The proof of Theorem~\ref{theorem:log-concave-solution} does not use the linear independence conditions of \cref{Thm:general-nonconvex}. A basic count of tight constraints is able to deliver the piecewise geometric structure, since the number of constraints in problem (DMP-LC) for a given support is small compared to the number of variables. Consider support $[1,n]$ in (DMP-LC). \cref{Thm:general-nonconvex} implies that $n - m$ of the $2n - 2$ constraints in \eqref{eq:dmp-bad-constraints}--\eqref{eq:dmp-strict} are tight. Since all constraint in \eqref{eq:dmp-strict} are strict (this is handled carefully in the proof) this implies all $n - m$ tight constraints are from \eqref{eq:dmp-bad-constraints}, which are of the form $x_{j-1}x_{j+1} \le x_j^2$. Setting $n-m$ of these constraints to equality directly yields the geometric structure \eqref{geometric-form}.


\subsection{Increasing failure rate}\label{ss:ifr-solution}

Recall the formulation (DMP-IFR') of the IFR moment problem in \Cref{ss:ifr-problem} with $y_j=\sum_{k=j}^n x_k$. We will show that the optimal solution has similar structure as the log-concave case, again using \cref{Thm:general-nonconvex,theorem:restrict-attention-to-extreme-points}.

Here we notice two facts. First, by the log-concave constraint and the non-increasing property of $y_j$, any feasible solution $y$ has a consecutive support naturally, and the support starts from $y_1 = 1$. This is different from the log-concave case. Second, if there is some $\ell$ such that $y_\ell = y_{\ell+1} > 0$, this combined with the constraint $y_{\ell-1}y_{\ell+1} \le y_{\ell}^2$ indicates that we have $y_{\ell-1} \le y_{\ell}$. However, we also have $y_{\ell-1} \ge y_{\ell}$ in the problem's constraints. This means that $y_\ell = y_{\ell+1} > 0$ implies $y_{\ell-1} = y_{\ell}$. Then by induction we have $y_1 = \dots = y_{\ell+1} = 1$.

Combine the two facts above, the interval $[1,n]$ can be divided into three consecutive parts: $[1,j_1),[j_1,j_2),[j_2,n]$, where we have $y_1 = \dots = y_{j_1} = 1$, $1 = y_{j_1} > \dots > y_{j_2} = 0$, $0 = y_{j_2} = \dots = y_{n} $, i.e., an all-one interval, a strictly decreasing interval, and an all-zero interval. Further, the optimal solution in the middle interval has a more detailed characterization stated here.

\begin{theorem}\label{Thm-IFR}
Every feasible instance of (DMP-IFR') has an optimal extreme point solution. Moreover, for every optimal extreme point solution $y^*$, there exist integers $1\le j_1\le j_2$ such that $y_j = 1$ when $j\le j_1$, $y_j = 0$ when $j\ge j_2$. The interval $[j_1, j_2]$ can be divided as follows. There exist (i) integers $u_i$ and $v_i$ for $i \in [1, m]$ with $j_1 = u_1 < v_1 = u_2 < v_2 \dots < v_{m-1} = u_m < v_m = j_2$ (ii) real parameters $\alpha_i > 0$, $0 < r_i < 1$ for $i \in [1,m]$ such that
\begin{equation}\label{ifr-geometric-form}
y^*_j =
\begin{cases}
\alpha_i r_i^{j - u_i} & \text{ for }  j \in [u_i, v_i]\\
0 & \text{ otherwise.}
\end{cases}
\end{equation}
\end{theorem}

We remark on an important difference in the analysis of the LC and IFR cases. Here it is not enough to have a lower bound on the number of tight constraints given by the first part of \cref{Thm:general-nonconvex}. The reason is that (DMP-IFR') has in the order of $2n$ constraints of type $f_p(x) \le 0$ (using the notation of \eqref{eq:reverse-convex-general}) corresponding to constraints \eqref{eq:ifr-prime-logconcave} and \eqref{eq:ifr-prime-monotonicity} in (DMP-IFR'), rather than $n$ such constraints in the LC case. This requires us to use the ``in addition'' part of \cref{Thm:general-nonconvex} that invokes the linear independence of gradients. For this reason, the proof of \cref{Thm-IFR} requires additional work.

\section{An implementation with numerical results}\label{s:computation-and-numerical}

In this section we results results in \Cref{s:characterize} to solve a representative sample of moment problem numerically. We focus on the moment problem over log-concave distributions with two moments as a proof of concept of our approach (these ideas carry over to the more general case). That is, we find an optimal solution to
\begin{subequations}\label{eq:DMP-LC-2moment}
\begin{align}
&&\max_{k, \ell: 1\le k\le \ell \le n} \ \max_{x \in \R^n} \  & \sum_{j=k}^\ell f_j x_j
\qquad\qquad\qquad\qquad\qquad\qquad\qquad\qquad\qquad\\
 && \text{s.t. } & \sum_{j=k}^\ell w_j^i x_j = q_i \text{ for } i \in [0, 2] \label{eq:2-moment-moments} \\
            && & x_{j-1}x_{j+1} \le x_j^2 \text{\ \ for } j \in (k, \ell) \label{eq:2-moment-lc}  \\
            && & x_j > 0 \text{ for } j \in [k, \ell]  \\
            && & x_j = 0 \text{ for } j \notin [k, \ell] \label{eq:2-moment-zero}
\end{align}
\end{subequations}
using the structure of optimal extreme point solutions in \cref{theorem:log-concave-solution}. According to that theorem, there exists an optimal piecewise geometric distribution for \eqref{eq:DMP-LC-2moment} with at most $m = 2$ pieces. Thus, we can restrict the search to finding feasible parameters $k$, $v_1$, $\ell$, $\alpha_1$, $\alpha_2$, $r_1$, and $r_2$ to construct an $x^*$ according to \eqref{geometric-form} that satisfies the constraints of the problem with the largest objective value. Observe that \eqref{geometric-form} captures the structure of constraints \eqref{eq:2-moment-lc}--\eqref{eq:2-moment-zero} in (DMP-LC), the choice of parameters is further restricted by the moment constraints \eqref{eq:2-moment-moments}.

A more traditional approach to solving \eqref{eq:DMP-LC-2moment} would be take $x$ as the decision variable and solve \eqref{eq:DMP-LC-2moment} directly. The resulting problem is nonconvex and (potentially) high-dimensional if $n$ is large, whereas our approach remains low-dimensional as $n$ grows.
%
%

\subsection{Computational approach}\label{ss:computational-method}

In this section we describe how to reduce the search for optimal extreme point solutions to \eqref{eq:DMP-LC-2moment} from a seven-dimensional decision space -- $k$, $v_1$, $\ell$, $\alpha_1$, $\alpha_2$, $r_1$, and $r_2$ -- to a four-dimensional decision space. The first three variables concerning the domain: $k$ and $\ell$ describe the support and $v_1$ the ``break-point'' between the two geometric pieces. The fourth parameter, which is denoted $\alpha$ in the sequel,  captures the geometric shape of the constraints and accounts for all of $\alpha_1$, $\alpha_2$, $r_1$, and $r_2$ when restricted to satisfy the moment constraints \eqref{eq:dmp-lc-moments}. We construct this parameter over the next several paragraphs.

The first step in this reduction is a normalization step. Recall that an instance of \eqref{eq:DMP-LC-2moment} is specified by the elements of the sample space $\Omega = (w_1, \dots, w_n)$ and the moments $q_1$ and $q_2$. For simplicity, we shift and scale the elements of the sample space so that the resulting distribution has mean $q'_1 = 0$ and variance $q'_2 = 1$. For each $j \in [1,n]$ subtract the mean $q_j$ from $w_j$ and scaling the result by $\epsilon := 1/\sqrt{q_2 - q_1^2}$. The resulting sample space is $\Omega' = (w'_1, \dots, w'_n) = \left\{w_1\epsilon - q_1 \epsilon, w_2 \epsilon - q_1 \epsilon, \dots, w_n \epsilon - q_1 \epsilon \right\}$. That is, $w'_j = w_j \epsilon - q_1 \epsilon$ for $j \in [1,n]$. Again we make the assumption as in \Cref{s:problem} that $w_j = j$ for simplicity, so that we have $w'_j = j\epsilon - q_1\epsilon$. 

Now, we fix the support $[k,\ell]$ and break point $v_1$. Our final algorithm will enumerate over these all possible values of $k$, $\ell$ and $v_1$ in an outer loop. Given $k$, $\ell$, and $v_1$, the remaining decision variables are $\alpha_1$, $\alpha_2$, $r_1$, and $r_2$. The zeroth moment condition amounts to
\begin{equation}\label{eq:where-we-start}
\sum_{j = k}^{v_1-1} \alpha_1 r_1^{j - k} + \sum_{j = v_1}^\ell \alpha_2 r_2^{j - v_1} = 1
\end{equation}
and similarly for the first and second moment conditions. In order to reduce the degrees of freedom further we manipulate the sums in \eqref{eq:where-we-start} and introduce some additional notation. First of all, we let $\rho := \alpha_2$ and observe that we can express $\alpha_1$ in terms of $\alpha_2$ and $r_1$. Indeed, since we have at the middle point $\rho := \alpha_2 = \alpha_1 r_1^{v_1-k}$, we have $\alpha_1 = \rho/r_1^{v_1 - k}$, in which case we can rewrite \eqref{eq:where-we-start} as
\begin{equation}\label{eq:re-index}
\sum_{j = 1}^{\tilde k} \frac{\rho}{r_1^j} + \sum_{j = 1}^{\tilde \ell} \rho r_2^j + \rho = 1
\end{equation}
where we re-index the sums and set $\tilde k = v_1 - k$ and $\tilde \ell = \ell - v_1$. The three terms in \eqref{eq:re-index} are the probability mass on the left, right, and at the middle point. Finally, for reasons that will become apparent below, we will set $r_2 := e^\alpha$ and $r_1 := e^{-\beta}$ for nonnegative scalars $\alpha$ and $\beta$ so that \eqref{eq:re-index} becomes
\begin{equation}\label{eq:zeroth-moment-condition}
\sum_{j = 1}^{\tilde k} \rho e^{\beta j} + \sum_{j = 1}^{\tilde \ell} \rho e^{\alpha j} + \rho = 1
\end{equation}
Figure~\ref{fig:figure-out-confusing-indexes} may assist the reader in tracking the notation in \eqref{eq:where-we-start}--\eqref{eq:zeroth-moment-condition}.

\begin{figure}[t]
\centering
\includegraphics[width=0.65\textwidth]{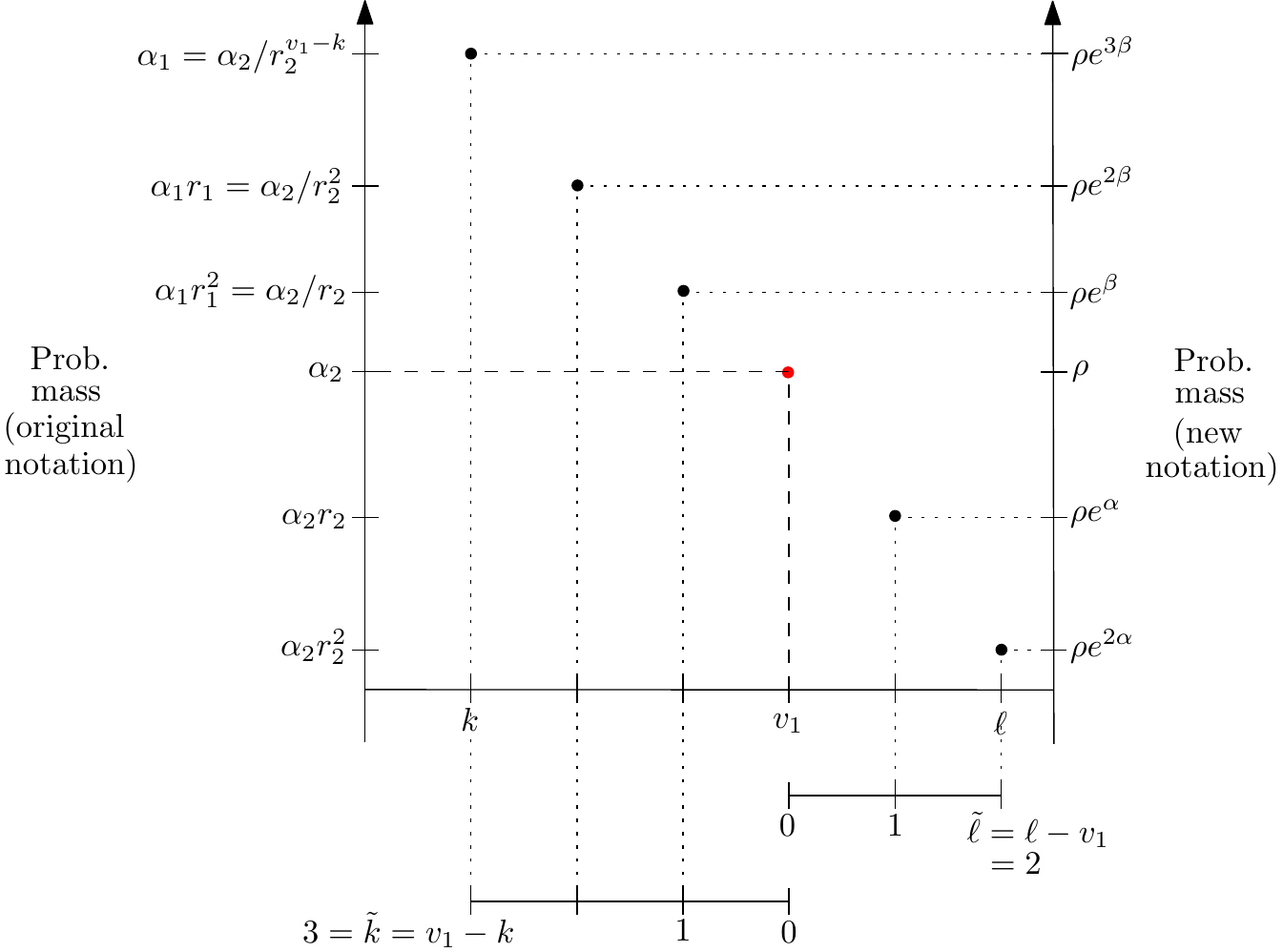}
\caption{Illustrating the notational development from \eqref{eq:where-we-start} to \eqref{eq:zeroth-moment-condition}.}
\label{fig:figure-out-confusing-indexes}
\end{figure}

For the first and second moments, we also define the $w'_j$ according to the indexing established in \eqref{eq:zeroth-moment-condition}. Moreover, we set $a = w'_{v_1} = v_1 \epsilon - q_1 \epsilon$, in which case the moment condition $\sum_{j = k}^\ell (w'_j)^ix_j = q_1'$ amounts to
\begin{equation}\label{eq:first-moment-condition}
\sum_{j = 1}^{\tilde k} \rho e^{\beta j} (a - j \epsilon) + \sum_{j = 1}^{\tilde \ell} \rho e^{\alpha j}(a + j \epsilon) + \rho a = 0.
\end{equation}
Similarly, the second moment condition is
\begin{equation}\label{eq:second-moment-condition}
\sum_{j = 1}^{\tilde k} \rho e^{\beta j} (a - j \epsilon)^2 + \sum_{j = 1}^{\tilde \ell} \rho e^{\alpha j}(a + j \epsilon)^2 + \rho a^2 = 1.
\end{equation}
Taken together, we have rephrased the problem to finding three unknowns -- $\alpha$, $\beta$ and $\rho$ -- in three equations \eqref{eq:zeroth-moment-condition}--\eqref{eq:second-moment-condition}. By first eliminating $\rho$, we get two equations in two unknowns:
\begin{align}
f(\alpha,\beta) &:=  \sum_{j=1}^{\tilde k} e^{\beta\, j}(a - j \epsilon) + \sum_{j=1}^{\tilde \ell} e^{\alpha \,j}(a + j \epsilon) + a = 0  \label{modified-moment-1}\\
g(\alpha,\beta) & := \sum_{j=1}^{\tilde k} e^{\beta\, j} + \sum_{j=1}^{\tilde \ell} e^{\alpha \,j} + 1  - \sum_{j=1}^{\tilde k} e^{\beta\, j}(a - j \epsilon)^2 - \sum_{j=1}^{\tilde \ell} e^{\alpha \,j}(a + j \epsilon)^2- a^2 = 0 \label{modified-moment-2}
\end{align}

The final step is to show that, given an $\alpha$, there is a unique choice of $\beta$ such that $f(\alpha, \beta) = 0$. Then, to identify common roots of $f(\alpha, \beta) = 0$ and $g(\alpha, \beta) = 0$ is equivalent to identifying the roots of a single equation $g(\alpha, h(\alpha)) = 0$, reducing the problem to a search for one unknown in one equation.

We achieve this final task by exploring monotonicity properties of $f$. First, a direct computation yields:
\begin{align*}
\frac{\partial f(\alpha,\beta)}{\partial \alpha} &= \sum_{j=1}^{\tilde \ell} j \cdot (a + j \epsilon)  e^{\alpha \,j} = \ex \left[\frac{X}{\rho} \cdot \frac{ X - a}{\epsilon}\cdot {\bone}_{X > a}\right] \\
\frac{\partial f(\alpha,\beta)}{\partial \beta} &= \sum_{j=1}^{\tilde k} j \cdot (a - j \epsilon)  e^{\beta \,j} = \ex \left[\frac{X}{\rho} \cdot \frac{a - X}{\epsilon}\cdot {\bone}_{X < a}\right].
\end{align*}
where $X$ is the discrete random variable with distribution $x = (x_1, \dots, x_n)$.
We then use the following technical lemma.

\begin{lemma}\label{lemma:technical-monotonic}
Suppose the polynomial $\phi(z) = \sum_{j=1}^{ M} a_j z^{i_j}$ with $z \in \R$ satisfies
\begin{equation}\label{coefficient-inequlity}a_1 \le a_2 \le \cdots \le a_{M}\quad \mbox{and}  \quad 1\le i_1 \le i_2 \le \cdots \le i_{M}.
\end{equation}
Then $\phi(z)$ has at most one root when $z > 0$ and is increasing on $\{ z \; | \; \phi(z) \ge 0 \}$.
\end{lemma}

It follows from~\eqref{eq:first-moment-condition} that $\ex \left[X \cdot {\bone}_{X > a}\right]=\sum_{j=1}^{\tilde k} e^{\alpha \,j}(a + j \epsilon) \ge 0$ and $\ex \left[X \cdot {\bone}_{X < a}\right] =\sum_{j=1}^{\tilde \ell} e^{\beta\, j}(a - j \epsilon) \le 0$. Now apply Lemma~\ref{lemma:technical-monotonic} to the polynomials $\sum_{j=1}^{\tilde k} e^{\alpha \,j}(a + j \epsilon)$ and $-\sum_{j=1}^{\tilde \ell} e^{\beta\, j}(a - j \epsilon) $ respectively (in the former, $z$ is $e^\alpha$ and $a_j = a + j \epsilon$). Supposing roots exist to these polynomials, define
\begin{equation*}
\alpha_0 =\min \{\alpha : \sum_{j=1}^{\tilde \ell} e^{\alpha \,j}(a + j \epsilon)=0 \}\;  \mbox{\ \ and \ \ }\;   \beta_0 = \min\{\beta : \sum_{j=1}^{\tilde k} e^{\beta\, j}(a - j \epsilon)=0 \}.
\end{equation*}
such that $\ex \left[X \cdot {\bone}_{X > a}\right] \ge 0$ if and only if $\alpha \ge \alpha_0$ and $\ex \left[X \cdot {\bone}_{X < a}\right] \le 0$ if and only if $\beta \ge \beta_0$. If roots do not exist set $\alpha_0 = -\infty$ and/or $\beta_0 = -\infty$. Thus it suffices to focus on the region where $\alpha \ge \alpha_0$ and $\beta \ge \beta_0$. As a result, when $\alpha \ge \alpha_0$,
\begin{equation*}
\ex \left[X \cdot  (X - a) {\bone}_{X > a}\right]= \left\{
\begin{array}{ll}
\ex \left[(X - a)^2 {\bone}_{X > a}\right]  + \ex \left[a \cdot  (X - a) {\bone}_{X > a}\right] \ge 0,  & \text{ if }\; a \ge 0\\
  \ex \left[ X^2 \,{\bone}_{X > a}\right]  - a \, \ex \left[X \cdot {\bone}_{X > a}\right] \ge 0, & \text{ if }\; a  < 0.
\end{array}
\right.
\end{equation*}
Similarly for $\beta \ge \beta_0$ we have
\begin{equation*}
\ex \left[X \cdot  (a - X) {\bone}_{X < a}\right]= \left\{
\begin{array}{ll}
  \ex \left[ - X^2 \,{\bone}_{X < a}\right]  + a \, \ex \left[X \cdot {\bone}_{X < a}\right] \le 0, & \text{ if }\; a \ge 0\\
 \ex \left[-(X - a)^2 {\bone}_{X < a}\right]  + \ex \left[a \cdot  (a - X) {\bone}_{X < a}\right] \le 0, & \text{ if }\; a  < 0.
\end{array}
\right.
\end{equation*}
In summary, we have $\frac{\partial f(\alpha,\beta)}{\partial \alpha} \ge 0$ when $\alpha \ge \alpha_0$ and $\frac{\partial f(\alpha,\beta)}{\partial \beta} \le 0$ when $\beta \ge \beta_0$.
This monotonicity yields our desired property that we can identify a mapping $h$ such that $\beta = h(\alpha)$.

To apply Newton's method to solve $g(\alpha, h(\alpha)) = 0$, must find the derivative with respect to $\alpha$. Observe that
\begin{eqnarray*}
\frac{\partial g(\alpha,\beta)}{\partial \alpha} = \sum_{j=1}^{\tilde \ell} j \cdot e^{\alpha \,j} - \sum_{j=1}^{\tilde \ell} j \cdot (a + j \epsilon)^2  e^{\alpha \,j} = \ex \left[\frac{1 - X^2}{\rho} \cdot \frac{ X - a}{\epsilon}\cdot {\bone}_{X > a}\right]
\end{eqnarray*}
and
\begin{eqnarray*}
\frac{\partial g(\alpha,\beta)}{\partial \beta} = \sum_{j=1}^{\tilde k} j \cdot e^{\beta\, j} - \sum_{j=1}^{\tilde k} j \cdot e^{\beta\, j}(a - j \epsilon)^2  = \ex \left[\frac{1 - X^2}{\rho} \cdot \frac{ a - X }{\epsilon}\cdot {\bone}_{X < a}\right],
\end{eqnarray*}
and so
\begin{eqnarray*}
\frac{\mathrm{d} g(\alpha, h(\alpha))}{\mathrm{d} \alpha} = \frac{\partial g(\alpha, \beta)}{\partial \alpha} + \frac{\partial g(\alpha, \beta)}{\partial \beta}\cdot \frac{\partial h(\alpha)}{\partial \alpha} = \frac{\partial g(\alpha, \beta)}{\partial \alpha} + \frac{\partial g(\alpha, \beta)}{\partial \beta}\cdot \frac{- \frac{\partial f(\alpha,\beta)}{\partial \alpha}}{\frac{\partial f(\alpha,\beta)}{\partial \beta}}.
\end{eqnarray*}
Using this derivative, Newton's method on the interval $[\alpha_0, +\infty]$ of real numbers  finds all roots of $g(\alpha, h(\alpha))$.

As a final note, when we get the solution of pairs of $\alpha, \beta$, we only include those that satisfy the inequality $\alpha + \beta \le 0$, which is translated from the log-concave constraint on the middle index $v_1$.

\subsection{Numerical results}\label{ss:numerical-results}

To illustrate the performance of the proposed computational approach, we implement it on a concrete example that appears in the literature \citep{subasi2009discrete}. We note that the main focus of the paper is the theoretical properties for global optima of shape-constrained discrete moment problems instead of developing fast algorithms. Therefore, we provide this example only for illustrative purposes. A more in-depth investigation of efficient computation methods for general problems will be left as future work.

In \cite{subasi2009discrete}, the authors aim at solving a specific discrete moment problem (Example~4) with log-concave constraint \eqref{eq:DMP-LC-2moment}. However, their methodology requires relaxing the constraint to be unimodal, which they solve via a linear programming. As a type of benchmark, we compare the bounds that can be derived by our method with theirs. In detail, the specific example we solve is \eqref{eq:DMP-LC-2moment} with data specified in \cref{table:numerical} below. The sample space (before scaling) is always the natural numbers up to $n-1$, i.e., $w_j = j - 1$.

Our benchmark calculations use the unimodal relaxation of \cite{subasi2009discrete}, described below in our notation.
\begin{subequations}
\begin{align*}
\max_{K} \ \max_{x \in \R^n} \  & \sum_{j=1}^n f_j x_j \\
 \text{s.t. } & \sum_{j=1}^n w_j^i x_j = q_i \text{ for } i \in [0, 2] \\
             & x_j \le x_{j+1} \text{ for } j \in [1, K-1]\\
             & x_j \ge x_{j+1} \text{ for } j \in [K, n-1].
\end{align*}
\end{subequations}
where $K$ is the ``mode'' of the distribution. Instead of moment constraints, we use the binomial moment constraints of \cite{subasi2009discrete}, i.e.
\begin{align*}
	\sum_{j = 1}^n \binom{w_j}{i} x_j = S_i,  \text{ for } i \in [0, 2]
\end{align*}
where the data $S_0,S_1,S_2$ can be transformed to moment data $q_0,q_1,q_2$ via the linear transformation: $q_0 = S_0$, $q_1 = S_1$, and $q_2 = 2S_2 + S_2$. Note that this linear transformation can be extended to higher moments, see \cite[Section~5.6]{prekopa2013stochastic} for details. The objective function is the probability mass on the positive values of $w_j$, i.e., $\Pr(X \geq 1)=\sum_{j = 1}^n f_j x_j = \sum_{j=2}^n x_j$ and provides an upper bound on the tail probability given the first two moments. Optimizing the negative of this objective also allows us to calculate lower bounds on tail probabilities. The results are shown in \cref{table:numerical}.

\begin{table}[!h]
\centering
\caption{Numerical result of the bounds for the total probability for non-negative values  with different constraints.}
\begin{tabular}{|l|l|l|l|l|l|l|}
\hline
\multicolumn{3}{|l|}{} & \multicolumn{2}{l|}{Unimodal} & \multicolumn{2}{l|}{Log-concave} \\ \hline
$n$    & $S_1$   & $S_2$   & LB              & UB          & LB              & UB             \\ \hline
5    & 1.9    & 1.3    & 0.8750          & 1           & 0.9000          & 1              \\ \hline
5    & 2.1    & 1.3    & 0.9750          & 1           & 0.9920          & 1              \\ \hline
5    & 1.9    & 1.7    & 0.8000          & 1           & 0.8094          & 0.8433         \\ \hline
11   & 5.2    & 13.1   & 0.9482          & 1           & 0.9684          & 1              \\ \hline
11   & 4.6    & 13.1   & 0.8745          & 1           & 0.8924          & 0.9026         \\ \hline
11   & 5.2    & 15.1   & 0.9208          & 1           & 0.9310          & 0.9921         \\ \hline
\end{tabular}
\label{table:numerical}
\end{table}

The LC constraint gives tighter lower and upper bounds in all cases. This is to be expected, since the unimodal relaxation is clearly a relaxation and so by solving the original log-concave version of the problem we are able to achieve tighter lower and upper bounds.


\section{Conclusion}\label{s:conclusion}

In summary, we use a reverse convex optimization approach to characterize optimal extreme point distributions for moment problems with reverse convex shape constraints. This characterization allowed us to design an exact low-dimensional algorithm for solving these problems to optimality.

There are several possible directions to apply and build on the results in this paper that we leave as future work. First, there are specific applications of robust optimization where log-concave or IFR distributions are common. One standard example is the robust newsvendor problem originally studied by \cite{scarf1958min} where having structural solutions to the second-stage moment problem can be useful in characterizing optimal inventory strategies. Second, although these results are for the discrete moment problem we believe there is scope to extend them through limiting arguments to the continuous case. Lastly, there is room to more deeply explore implementations of our computational approach that pays attention to issues of numerical stability and scaling properties.




\bibliographystyle{plainnat}
{\small

}

\appendix

\section{Appendix: Technical proofs}

\subsection{Proof of \cref{prop:equivalent-formulations-lc}.}

Setting $u = v = 1$ in \eqref{eq:dmp-good-constraints} specializes to \eqref{eq:dmp-bad-constraints}. Further, \eqref{eq:dmp-good-constraints} guarantees a consecutive support: if there exist $j_1<j_2<j_3$ such that $x_{j_1},x_{j_3} > 0$, $x_{j_2} = 0$, by setting $u = j_2 - j_1, v = j_3 - j_2$, the constraint is $x_{j-u}^v x_{j+v}^u \le x_j^{u+v}$ violated. Hence every feasible distribution of (DMP-LC') is a feasible distribution of (DMP-LC) with the same objective value (note that the objectives of both problems are identical).

On the order hand, any feasible distributions to problem \eqref{eq:DMP-LC} with support $[k,\ell]$ satisfies \eqref{eq:dmp-good-constraints} and by a straightforward induction starting with \eqref{eq:dmp-bad-constraints} as a base case we can argue that $ x_{j-u}^v x_{j+v}^u \le x_j^{u+v}$ holds for $j \in (k, \ell), j-u \ge k, j + v\le \ell$.%
%
%
\footnote{To give a concrete example, we show how to derive the inequality $x_j^3 \ge x_{j-2} x_{j + 1}^2$ ($u = 2$ and $v = 1$) starting from \eqref{eq:dmp-bad-constraints}. From \eqref{eq:dmp-bad-constraints} we have the two constraints: $x_j^2 \ge x_{j-1}x_{j+1}$ and $x_{j-1}^2 \ge x_{j-2}x_j$. Dividing the left-hand side of the former by the right-hand side of the latter (and vice versa) yields the inequality $x_jx_{j-1} \ge x_{j-2}x_{j+1}$. Hence, starting from $x_j^2 \ge x_{j-1}x_{j+1}$ and multiplying both sides by $x_j$ yields: $x_j^3 \ge x_j x_{j-1}x_{j+1} \ge x_{j-2}x_{j+1}^2$, as required.}
%
%
 For those points such that $j-u$ or $j+v$ is outside the support, or the middle point $j$ outside the support, the constraint $ x_{j-u}^v x_{j+v}^u \le x_j^{u+v}$ holds naturally since the left hand side is zero for these cases. In other words, \eqref{eq:dmp-good-constraints} is satisfied. Hence every feasible distribution of (DMP-LC) is a feasible distribution of (DMP-LC') with the same objective value.

\subsection{Proof of \cref{prop:set-convex-LC}.}

We first prove a preliminary lemma for establishing \cref{prop:set-convex-LC}.

\begin{lemma}\label{lemma:set-convex-LC} The set $\{ (x,y,z): x^{u} y^{v} \ge z^{u + v}, x \ge 0, y \ge 0, z \ge 0 \}$ is convex for any positive integers $u$ and $v$.
\end{lemma}

\begin{proof}[Proof of \cref{lemma:set-convex-LC}]
For any integers $u$ and $v$, let $t$ be an integer such that $u + v \le 2^t$. From point $11$ on page $95$ of \cite{ben2001lectures}), the set $\{ (x_1,\cdots, x_{2^t},z) : x_j \ge 0, \; 1\le j \le 2^t,\; z \le (\Pi_{j=1}^{2^t} x_j)^{1/2^t} \}$ is conic-quadratic representable, and thus convex. Therefore, when intersecting with linear constraints, the set
\begin{eqnarray*}
&S : = \{ (x,y,w) : x_j \ge 0, \; 1\le j \le 2^t,\; 0 \le z \le (\Pi_{j=1}^{2^t} x_j)^{1/2^t},\; x_j = x, \; 1\le j \le u,\\
& x_j = y, \; u+1\le j \le u+v, \; x_j = z, \; u+v+1 \le j \le 2^t    \}
\end{eqnarray*}
is convex as well. When $(x,y,z) \in S$, $z \le (\Pi_{j=1}^{2^t} x_j)^{1/2^m} $ is equivalent to $z \le (x^{u} y^{v} z^{(2^t - u - \ell)})^{1/2^t} $, which can be further rewritten as $  z^{u + v} \le x^{u} y^{v}$. Consequently, it is straightforward to verify that $S = \{ (x,y,z) : x^{u} y^{v} \ge z^{u + v}, x \ge 0, y \ge 0, z \ge 0 \}$, and the conclusion follows.
\end{proof}

\begin{proof}[Proof of \cref{prop:set-convex-LC}]

First observe that

\begin{equation*}
S:=\{ (x,y,z) : x^{u} y^{v} > z^{u + v},\; x \ge 0, y \ge 0, z \ge 0 \} = \bigcup_{\epsilon>0} S(\epsilon),
\end{equation*}
where $S(\epsilon) = \{ (x,y,z) : x^{u} y^{v} \ge (z+\epsilon)^{u + v},\; x \ge 0, y \ge 0, z \ge 0 \}$. Then for any $(x_1,y_1,z_1)$ and $(x_2,y_2,z_2) $ in $S$, there exist $\epsilon_1$ and $\epsilon_2$ such that $(x_1,y_1,z_1) \in S(\epsilon_1)$ and $(x_2,y_2,z_2) \in S(\epsilon_2)$. Without loss of generality, we assume that $\epsilon_1 \ge \epsilon_2$. As a result, we have that $S(\epsilon_1) \subset S(\epsilon_2) $ and $(x_1,y_1,w_1) \in S(\epsilon_2)$. Moreover, according to \cref{lemma:set-convex-LC}, $S(\epsilon_2)$ is a convex set. That is $\alpha (x_1,y_1,z_1) + (1 - \alpha)(x_2,y_2,z_2) \in S(\epsilon_2) \subset S$ for any $0 \le \alpha \le 1$. Therefore, since the union of convex sets is convex, $S$ is convex as desired.
\end{proof}

\subsection{Proof of \cref{lemma: IFR-tail-definition}.}
According to \cref{definition:ifr}, we have the following inequality if $x$ is an IFR distribution:
\begin{equation*}
\tfrac{x_j}{\sum_{k=j}^n x_k} - \tfrac{x_{j+1}}{\sum_{k=j+1}^n x_k} \le 0, \text{ for } j \in [1,n-1].
\end{equation*}
This is equivalent to
\begin{align}\label{eq:tail-equivalent-1}
x_j \sum_{k=j+1}^n x_k - x_{j+1} \sum_{k=j}^n x_k \le 0, \text{ for } j \in [1,n-1].
\end{align}
While if $\{\bar F_1,\dots, \bar F_n\}$ is log-concave, we have
\begin{equation*}
\sum_{k=j-1}^n x_k \sum_{k=j+1}^n x_k -\left(\sum_{k=j}^n x_k\right)^2 \le 0, \text{ for } j \in [2,n-1].
\end{equation*}
This is equivalent to
\begin{align}\label{eq:tail-equivalent-2}
x_j \sum_{k=j+1}^n x_k - x_{j+1} \sum_{k=j}^n x_k \le 0, \text{ for } j \in [1,n-2].
\end{align}
Inequality \eqref{eq:tail-equivalent-1} and \eqref{eq:tail-equivalent-2} are exactly the same except that \eqref{eq:tail-equivalent-2} does not include the case where $j = n-1$. In this case the inequality holds naturally: $x_{n-1}x_n-x_n(x_{n-1}+x_n)\le 0$. Thus the two definitions of IFR distribution are equivalent.

\subsection{Proof of \cref{theorem:restrict-attention-to-extreme-points} and \cref{lemma:extreme-points-in-set}.}
\label{sec:proof_restrict-attention-to-extreme-points}

\begin{proof}[Proof of \cref{lemma:extreme-points-in-set}]

The fact that $\ext \cch S \subseteq S$ follows immediately from \cite[Theorem~3.5]{klee1957extremal}. Suppose, by way of contradiction, that there exists an $x \in \ext \cch S$ that is not an extreme point of $S$. Then there exists $y, z \in S$ with $y \neq z$ such that $x = \lambda y + (1- \lambda) z$ where $\lambda > 0$. However, since $y,z \in \cch S$ this contradicts that $x \in \ext \cch S$. The result then holds.

\end{proof}

With \cref{lemma:s-convex}, \cref{lemma:closed-convex-full-compact}, and \cref{lemma:extreme-points-in-set} in hand, we can now establish \cref{theorem:restrict-attention-to-extreme-points}.

\begin{proof}[Proof of \cref{theorem:restrict-attention-to-extreme-points}]
The problem $\min \left\{ c^\top  x: x \in \cch S\right\}$ has an optimal extreme point solution $x^* \in \ext \cch S$ by \cref{lemma:s-convex} and the fact that $\cch S$ is a compact convex set by \cref{lemma:closed-convex-full-compact}. Since $S \subseteq \cch S$ we know $\min \left\{ c(x) : x \in \cch S\right\} \le \min \left\{ c(x) : x \in  S\right\}$. However, since $x^* \in S$, by \cref{lemma:extreme-points-in-set} we have $
c(x^*) = \min \left\{ c(x) : x \in \cch S\right\} \le \min \left\{ c(x) : x \in S \right\} \le c(x^*)$, since $x^*$ is optimal to the minimization over $\cch S$ and feasible to the minimization over $S$. However, this means all inequalities must be equalities and so $\min \left\{ c(x) : x \in S\right\} = c(x^*)$. Since $x^* \in \ext S$ by \cref{lemma:extreme-points-in-set}, this implies \eqref{eq:linear-optimization-problem} has an optimal extreme point solution.
\end{proof}

\subsection{Proof of \cref{lemma:boundary-conditions}.}
\label{sec:proof_boundary-conditions}
Since $c$ is lower-semicontinuous and quasiconcave and $X$ is compact, \cref{theorem:restrict-attention-to-extreme-points} implies that there exists an optimal extreme point solution $x^*$. Let $C_p = \R^n \setminus R_p$. Then $C_p$ is an open convex set, since $R_p$ is closed and reverse convex. Since $x^* \in X$ then $x^* \notin C_p$ for all $p$. For all $p$, let $y_p$ be such that $\dist(x^*, \cl (C_p)) = \dist(x^*, y_p)$;
that is, $y_p$ minimizes the distance between $x^*$ and the closure of $C_p$. Note that $\bd(C_p)= \bd(R_p)$, by definition. Clearly, $y_p \in \bd(R_p)$, for all $p$. If $\dist(x^*, \cl (C_p)) = 0$ then take $y_p = x^*$. In this case, $x^* \in \bd(R_p)$.

Using the vector $y_p$ we can define for all $p \in [1,P]$ a supporting hyperplane of $\cl(C_p)$ with normal $\alpha_p$ and right-hand side $\beta_p$ that weakly separates $C_p$ from the point $x^*$. These hyperplanes define the polyhedron $\hat X= \{ x : \alpha_p^\top x \ge \beta_p, \text{ for } p \in [1,P]\}$ that is incribed in $X$. In the special case that $y_p \neq x^*$, the hyperplane $\{ x : (x^* - y_p)^\top (x - y_p) \le 0 \}$ does the trick, by the standard projection theorem. Note that $x^* \in \hat X$ and, moreover, $\hat X \subseteq X$. Indeed, since $\alpha_p^\top x \le \beta_p$ is a supporting hyperplane of $\cl(C_p)$ then, the set of $x$ that satisfy $\alpha_p^\top x \ge \beta_p$ lie on the boundary of $C_p$ or outside of $C_p$. Such an $x$ lies entirely inside of $R_p$. This implies $\hat X$ is a subset of $R_p$ for all $p$, and so $\hat X \subseteq X$.

Consider the optimization problem
\begin{align}\label{eq:linear-problem}
\begin{split}
\min \ &  c(x) \\
s.t. \ & \alpha_p^\top x \ge \beta_p  \text{ for } p \in [1,P].
\end{split}
\end{align}
Since $x^* \in \hat X \subseteq X$ and $x^*$ is an optimal solution of the original problem, $x^*$ is optimal solution of \eqref{eq:linear-problem}. Moreover, $x^*$ is an extreme point of $\hat X$ and so at least $n$ linearly independent tight constraints at $x^*$, by the characterization of extreme points of polyhedra \cite[Theorem~2.3]{bertsimas1997introduction}. Hence at least $n$ of the inequalities $\alpha_p^\top x \ge \beta_p$ must be tight at $x = x^*$. The points in $\hat X$ that satisfy $\alpha_p^\top x = \beta_p$ are boundary points of $R_p$. Hence, $x^*$ lies on the boundary of at least $n$ of the sets $R_p$.
\endproof


\subsection{Proof of Claim~\ref{complete-separation-Xp} in the Proof of \cref{Thm:general-nonconvex}.}
\label{sec:proof_claim-complete-seperation}
We employ the following two subclaims.

\begin{subclaim}\label{claim:y-convex-open}
The set $Y_p$ is a convex and open set for $p=1,\ldots, P$.
\end{subclaim}

\noindent \emph{Proof of \cref{claim:y-convex-open}:}  By assumption, $S_1 := \{ x: f_p(x) > 0, x \ge 0 \}$ is convex. Therefore, $S_2 := \{x: x >= 0 , f(x)>0\} \cap \{ x : Ax = b \}$ is also a convex set since we are interesting $S_1$ with the convex set $\{x : Ax = b\}$. Moreover, the set $S_3 :=  S_2 \cap \{ x : x_S > 0, x_{\bar S} = 0\}$ is again convex since $\{ x : x_S > 0, x_{\bar S} = 0\}$ is a convex set. Finally, consider the affine map $y \mapsto (By + x^*S, 0_{\bar S})$. Note that $Y_p$ is the inverse image of this map and is therefore convex.


Moreover, for any $y_1 \in Y_p$, let
\begin{equation*}
0 < \delta = \min \{  (B y_1 + x^*_S)_j, j=1,\ldots,|S| : f_p(B y_1 + x^*_S; 0) >0 \}.
\end{equation*}
Since $f_p(\cdot)$ is continuous, there exists an $\epsilon > 0$ such that for any $\| y - y_1 \|_2 \le \epsilon$ we have
\begin{equation*}
\min \{  (B y + x^*_S)_j, j=1,\ldots,|S| : f_p(B y + x^*_S; 0) >0  \} \ge \delta/2 >0.
\end{equation*}
Thus $y \in Y_p$ and $Y_p$ is open. This completes the proof of \cref{claim:y-convex-open}. \qquad $\dagger$
\vskip 5pt
Moreover, we have a ``strong separation property'' of $Y_p$ described as follows.
\begin{subclaim}\label{strong-separation-Yp}
There exists a $d_p \neq 0$ and $\hat \beta_p >0$ such that
\begin{equation}\label{eq:separation-Yp}
\left\{ \begin{array}{rl}
d_p^{\top} y \ge \hat{\beta}_p > 0, \text{ for } y \in Y_p & \mbox{if}\quad 0 \not\in \cl(Y_p) \\
d_p^{\top} y > 0, \text{ for } y \in Y_p & \mbox{if}\quad 0 \in \cl(Y_p)
\end{array}\right.
\end{equation}
Moreover, letting $g_p(y) = f_p(By + x^*_S;0)$ and assuming $\nabla g_p(0) \neq 0$, if $0 \in \cl(Y_p)$ then $\nabla g_p(0)^{\top} y > 0$ for all $y \in Y_p$.
\end{subclaim}
\noindent \emph{Proof of \cref{strong-separation-Yp}:}
Note that $f_p(B\cdot 0 + x^*_S; 0) =  f_p(x^*_S; 0) = f_p(x^*) \le 0$, thus $0 \not\in Y_p$.
Since by \cref{claim:y-convex-open} $Y_p$ is convex, $\cl(Y_p)$ is both closed and convex. Then when $0 \not\in \cl(Y_p)$, by the strong separation theorem for closed convex sets (see, for instance, \cite[Corollary~5.80]{hitchhiker}), there exist $d_p \neq 0$ and $\hat{\beta}_p >0$ such that $d_p^{\top} y \ge \hat{\beta}_p > 0$ for $y \in Y_p$. In the case of $0 \in \cl(Y_p)$, weak separation holds; that is there exists an $\hat \alpha_p \neq 0$ such that $\hat \alpha_p^{\top} y \ge 0^{\top} y  = 0$ for $y \in Y_p$. Together this yields \eqref{eq:separation-Yp}.

To establish the ``moreover'', note that $g_p(y) = f_p(By + x^*_S;0) \ge 0$ for any $y \in \cl(Y_p)$. Hence, $g_p(0) = f_p (x^*_S;0)=f_p(x^*) \le 0$. Combining these two facts gives that $g_p(0) = 0$ when $0 \in \cl(Y_p)$. That is, $0$ is a global minimizer of the problem
\begin{equation*}
\begin{array}{ll}
\min & g_p(y)\\
\mbox{s.t.}& y \in \cl(Y_p).
\end{array}
\end{equation*}
Thus the following optimality condition in the form of variational inequality holds: $\nabla g_p(0)^{\top} (y - 0) \ge 0$ for $y \in \cl(Y_p)$, which trivially leads to $\nabla g_p(0)^{\top} y \ge 0$ for $y \in Y_p$. Since $Y_p$ is open, we get strict separation $\nabla g_p(0)^{\top} y > 0$ for all $y \in Y_p$. This completes the proof of \cref{strong-separation-Yp}. \qquad $\dagger$ \vskip 5pt
We are now ready to prove \cref{complete-separation-Xp}. We show that \eqref{general-separation} holds with $\hat \alpha_p = ( B(B^{\top}B)^{-1}d_p; \gamma_p)$ with $d_p$ being defined in \cref{strong-separation-Yp} and any $\gamma_p \in \BR^{n - |S|}$ and $\hat \beta_p$ as constructed in \cref{strong-separation-Yp}. Indeed, for any $x \in X_p$, due to \eqref{transformation}, we can find a $y \in Y_p$ such that $x = (B y + x^*_S;0) = (By;0)+x^*$.
Consequently,
\begin{equation*}
\hat \alpha_p^{\top} (x - x^*)  =d_p^{\top} (B^{\top}B)^{-\top}B^{\top}By + \gamma_p^{\top} 0= d_p^{\top}y.
\end{equation*}
Then according to \cref{strong-separation-Yp}, \eqref{general-separation} holds.

To establish the ``moreover'' of \cref{complete-separation-Xp}, observe that when $\nabla g_p(0) \neq 0$ and $0 \in \cl(Y_p)$, by letting $\hat \alpha_p = ( B(B^{\top}B)^{-1}\nabla g_p(0); \gamma_p)$ with any $\gamma_p \in \BR^{n - |S|}$, we have $\hat \alpha_p^{\top} (x - x^*) > 0$ for $x \in X_p$. The argument here is analogous to what we used when establishing \eqref{general-separation}.

Now, suppose $[ \nabla f_p(x^*) ]_S \not\in \mathcal{L}(S) $
and $0 \in \cl(Y_p)$. We argue that
\begin{equation}\label{gradient-separation}
\nabla f(x^*)^{\top}(x - x^*)  > 0, \text{ for }  x \in X_p.
\end{equation}
First a direct computation yields
\begin{equation}\label{gradient-Gp}
\nabla g_p(0) = [B^{\top}\,0]\, \nabla f_p(B\, y + x^*_{S}; 0 )\big{|}_{y=0} = [B^{\top}\,0]\, \nabla f_p(x^*_{S}; 0 ) = [B^{\top}\,0]\,\nabla f_p(x^*) = B^{\top} [\nabla f_p(x^*)]_S .
\end{equation}
Since $[\nabla f_p(x^*)]_S \not\in \mathcal{L}(S) $, we have $\nabla g_p(0) \neq 0$. Otherwise,  due to \eqref{gradient-Gp} $[ \nabla f_p(x^*_S) ]_S$ belongs to the null space of $B^{\top}$, which is exactly $\mathcal{L}(S)$, giving rise to a contradiction.
For any $x \in X_p$, $A_S \, (x_S - x^*_S) = A_S \, x_S - A_S\, x^*_S = 0$, thus $x_S - x^*_S \in \mbox{Null}(A_S)$. Moreover, recall that the columns of $B$ span the whole $\mbox{Null}(A_S)$; then there exists a
$\theta \neq 0$ such that $x_S - x^*_S  = B \, \theta$. Now let $\hat \alpha_p = \left(B (B^{\top}B)^{-1} B^{\top} \nabla f_p(x^*_S) ; \gamma_p\right)$ with any $\gamma_p \in \BR^{n - |S|}$. According to \eqref{general-separation} we have
\begin{eqnarray*}
\nabla f(x^*)^{\top}(x - x^*) &=& (\nabla f(x^*)-\hat \alpha_p)^{\top}(x - x^*)  + \hat \alpha_p^{\top}(x - x^*) \\
& = & \left( [\nabla f(x^*)]_S - B(B^{\top}B)^{-1}B^{\top}\, [\nabla f_p(x^*)]_S \right)^{\top}(x_S - x^*_S)  + \hat \alpha_p^{\top}(x - x^*) \\
& = &  [\nabla f_p(x^*)]_{S}^{\top} \left( I - B (B^{\top}B)^{-1}B^{\top}\right) B \, \theta+ \hat \alpha_p^{\top}(x - x^*) \\
& = & \hat \alpha_p^{\top}(x - x^*)  >0.
\end{eqnarray*}
Thus \eqref{gradient-separation} holds, completing the proof of \cref{complete-separation-Xp}.

\subsection{Proof of Theorem \ref{Thm-IFR}.}
\label{sec:proof_Thm-IFR}
Since $1 = y_1 \ge \dots \ge y_n \ge 0$, the feasible region of (DMP-IFR') is closed and bounded and so Theorem~\ref{theorem:restrict-attention-to-extreme-points} implies there exists an optimal extreme point solution $y^*$. The existence of the three subintervals $[1,j_1),[j_1,j_2),[j_2,n]$ for $y^*$ is argued above the theorem and let $[1,k]$ be the support of $y^*$ (here $k = j_2 - 1$). Consider the following problem:
\begin{subequations}\label{eq:DMP-IFR-sub}
\begin{align}
\max_{y \in \R^k} \  & \sum_{j=1}^n f_j (y_j - y_{j+1}) \\
\text{s.t. } & \sum_{j=1}^n (w_j^i-w_{j-1}^i) y_j = q_i \text{\ \ for } i \in [0, m]  \\
             & y_{j-1}y_{j+1} \le y_j^2 \text{\ \ for } j \in (1, k) \label{eq:sub-logconcave} \\
             & y_j - y_{j+1} \ge 0 \text{\ \ for } j \in [1, k) \label{eq:sub-monotonicity} \\
             & y_j \ge 0 \text{\ \ for } j \in [1, k] \label{eq:sub-nonnegativity},
\end{align}
\end{subequations}
which is the subproblem of (DMP-IFR') throwing away the last $n - k$ indices. It is easy to verify that every feasible solution of \eqref{eq:DMP-IFR-sub} with padding $n - k$ many 0's in the end is also feasible for (DMP-IFR'), and $y^*$ with its $k$ nonzero element is feasible to \eqref{eq:DMP-IFR-sub}. For simplicity, we denote this truncated $y^*$ as $y^*$ in the following argument. We conclude here that $y^*$ is also an optimal extreme point solution for \eqref{eq:DMP-IFR-sub}. Note that none of the nonnegativity constraints \eqref{eq:sub-nonnegativity} are tight at $y^*$ since $[1,k]$ is its support.

We now apply \cref{Thm:general-nonconvex} to \eqref{eq:DMP-IFR-sub} to describe the structure of the optimal extreme point solution $y^*$. First, we must verify the conditions of the theorem. The log-concavity constraints \eqref{eq:sub-logconcave} are reverse convex w.r.t. $\R^n$ by \cref{prop:set-convex-LC}, and the monotonicity constraints \eqref{eq:sub-monotonicity} are linear, thus also reverse convex. Here we assume that $m+1 < k$ otherwise the problem is trivial by solving linear equations. Thus by the theorem there are at least $k-m-1$ tight constraints among \eqref{eq:sub-logconcave} and \eqref{eq:sub-monotonicity}. To further refine this conclusion we verify the conditions of the ``moreover'' part of Theorem~\ref{Thm:general-nonconvex}. To do so, we look at the gradients of the constraint functions. We want to argue that the gradients of the inequality constraints are not in the space spanned by the gradients of the equality constraints. That is, let $f_p(y) =  y_{p-1}y_{p+1} - y_p^2$ for $p \in (1,k)$ and $h_q(y) = y_{q+1} - y_q$ for $q \in [1,k)$\footnote{We use the index $p,q$ instead of $j$ to conform with the statement of \cref{Thm:general-nonconvex}.} and set $a^i := (w_1^i - w_0^i,\dots,w_k^i - w_{k-1}^i)^\top$ for $i \in [0,m]$ and $A$ the matrix with rows corresponding to the $a^i$. We want to verify that
\begin{equation}\label{eq:check-not-in-span}
    \nabla f_p(y^*), \nabla h_q(y^*) \notin \mathcal L := \text{span} (a^0,a^1,\dots,a^m), \text{ for all } p \in (1,k) \text{ and } q \in [1,k).
\end{equation}
Further, by calculation we have
\begin{align}\label{eq:gradient-fy}
 [ \nabla f_p(y) ]_j = \left\{ \begin{array}{ll} y_{p+1} & \mbox{if}\;  j =p - 1  \\ -2y_{p}  & \mbox{if}\;  j =p   \\ y_{p-1} & \mbox{if}\;  j =p +1  \\ 0 & \mbox{otherwise}. \end{array} \right.
\end{align}
and
\begin{align}\label{eq:gradient-hy}
 [ \nabla h_q(y) ]_j = \left\{ \begin{array}{ll} -1  & \mbox{if}\;  j =q   \\ 1 & \mbox{if}\;  j =q +1  \\ 0 & \mbox{otherwise}. \end{array} \right.
\end{align}
To show these gradients are not in the span of $\mathcal L = \text{span}\left\{a^0, \dots, a^m\right\}$ we first of all construct a vector $s$ that is perpendicular to all the vectors $a^i$ for $i \in [0,m]$; that is, $s \in \mathcal L^\perp$. Then we show that for all $p \in (1,k)$ and $q \in [1,k)$, $\nabla f_p(y^*)$ and $\nabla h_p(y^*)$ have a nonzero inner product with $s$. This allows to conclude \eqref{eq:check-not-in-span} and thus the conclusion of Theorem~\ref{Thm:general-nonconvex}.

In order to construct $s$ we start with something simpler. Define a vector $v \in \BR^k$ as follows
\begin{equation*}
v_j := (-1)^j\tbinom{k}{j}, \text{ for } j \in [1,k].
\end{equation*}
The following claim, whose proof is found in the next subsection, demonstrates that $v$ is orthogonal to $w^i$, where $w^i = (w_1^i,\dots,w_k^i)^\top$ for $i \in [0,m]$.
\begin{claim}\label{orthogonal-claim}
$v^{\top} w^i =  \sum_{j=1}^{k} (-1)^j \tbinom{k}{j} (w_j)^i  = 0 \text{ for }  i \in [0, m]$.
\end{claim}
In other words, when defining the matrix $W \in \R^{m+1 \times k}$ whose rows are the $w^i$'s, from \cref{orthogonal-claim} we have $Wv = 0$. Note that the matrix $A$ has the form $A = W(I - U)$ where $I$ is the identity matrix and $U$ is the ``upper diagonal'' matrix with $U_{p,p+1} = 1, p \in [1,k]$ and the rest of it elements $0$. Thus we can find our target value $s$ that solves $As = 0$ by setting $s$ equal to the solution of $(I-P)s = v$; namely, $s_p = \sum_{j = p}^k (-1)^j \tbinom{k}{j}$ for $p \in [1,k]$. We have found a vector $s$ in the orthogonal space $\mathcal L^\perp$. Straightforward computation then shows
\begin{equation*}
    s^\top  \nabla f_p(y^*) = y^*_{p -1} s_{p+1} -2 y^*_p s_p + y^*_{p+1} s_{p-1} \text{ \ \ and  \ \ }
    s^\top  \nabla h_q(y^*) = s_{q+1} - s_q
\end{equation*}
for $p \in (1,k), q \in [1,k)$. From the fact that $y_j^* > 0 $ and that the sign of the $s_p$ alternates according to $p$ -- i.e., $s_p s_{p+1} <0 ,p \in [1,k)$ -- we conclude that the two inner product above are nonzero for all $p$. This implies that $\nabla f_p(y^*) \not\in \mathcal{L}$. And similarly $ \nabla h_q(y^*) \not\in \mathcal{L}$.



Now all the conditions in \cref{Thm:general-nonconvex} are satisfied. Thus we must have at least $k-m-1$ tight constraints whose gradients are linear independent. Recalling that the nonnegativity constraints are not tight, we isolate our attention to the constraints \eqref{eq:sub-logconcave}--\eqref{eq:sub-monotonicity} indexed in $[1,j_1)$ and $[j_1,k]$.

For the constraints  $ y_{j-1}y_{j+1} \le y_j^2, y_j - y_{j+1} \ge 0$ in $[1,j_1)$, both of them are tight since $y_j$ are all one in this interval. From the computation of the gradients of $f(y)$ and $h(y)$ in \eqref{eq:gradient-fy} and \eqref{eq:gradient-hy}, they form a set with at most $j_1 - 1$ many of them are linear independent, i.e. the constraints in this interval can give at most $j_1 - 1$ many tight constraints whose gradients are linear independent.

For the point $j = j_1$, both the constraints $ y_{j-1}y_{j+1} \le y_j^2$ and $y_j - y_{j+1} \ge 0$ in $[1,j_1)$ are not tight since we have $y_{j_1 - 1} = 1 = y_{j_1 } > y_{j_1+1}$.
For the interval $(j_1,k)$, we have argued that the constraints $ y_j - y_{j+1} \ge 0$ cannot be tight since $y_j$ are strictly decreasing in this region. By a simple calculation we must have at least $k-m-1 - (j_1 - 1) = k - m - j_1$ many constraints in the form $y_{j-1}y_{j+1} \le y_j^2$ are tight here. That is to say, among the $k - j_1 - 1$ many constraints in this form, at most $k - j_1 - 1 - (k - m - j_1) = m- 1$ many constraints are not tight. From here, similar reasoning to the proof of the log-concave case in Theorem~\ref{theorem:log-concave-solution} yields the form \eqref{ifr-geometric-form}. Further details are omitted.
%

\subsubsection{Proof of Claim~\ref{orthogonal-claim}.}

In fact, let $n = k$, we can prove a stronger result: $\sum_{j=1}^{n}  (-1)^j \tbinom{n}{j} (w_j + \delta )^i = 0$ for any $\delta$ and $i \in [0, n -1]$. We shall prove this identity by mathematical induction on $n$. When $n =1$, $j$ has only one choice  $0$ and for any $\delta$
\begin{equation*}
\sum_{j=1}^{n} (-1)^j \tbinom{n}{j} (w_j + \delta )^0= \sum_{j=1}^{n} (-1)^j \tbinom{n}{j} = (1 - 1)^n = 0.
\end{equation*}

Now suppose this is true for $n$, and let's verify the validness of desired identify for  $n + 1$. Recall the combinatorial identity $\tbinom{n+1}{j} = \tbinom{n}{j} + \tbinom{n}{j-1}$. Then for any $\delta$ and $i \le n$, we have
\begin{eqnarray*}
\sum_{j=1}^{n+1}(-1)^j \tbinom{n+1}{j} (w_j + \delta)^i &=& \sum_{j=1}^{n+1}(-1)^j (\tbinom{n}{j} + \tbinom{n}{j-1})(w_j + \delta)^i \\
&=& \sum_{j=1}^{n+1}(-1)^j \tbinom{n}{j} (w_j + \delta)^i + \sum_{j=1}^{n+1}(-1)^j \tbinom{n}{j-1}  (w_j + \delta)^i \\
&=& \sum_{j=1}^{n}(-1)^j \tbinom{n}{j} (w_j + \delta)^i + \sum_{j=1}^{n+1}(-1)^j \tbinom{n}{j-1}  (w_j + \delta)^i \\
&=& \sum_{j=1}^{n}(-1)^j \tbinom{n}{j} (w_j + \delta)^i + \sum_{j=1}^{n}(-1)^{j+1} \tbinom{n}{j} (w_{j +1}+ \delta)^i \\
&=& \sum_{j=1}^{n}(-1)^j \tbinom{n}{j} \left( (w_j + \delta)^i -  (w_{j+1} + \delta)^i \right),
\end{eqnarray*}
where the third equality is due to $\tbinom{n}{n+1} = 0$, and the fourth equality follows by renaming $j-1$ as the new index and from the fact $\tbinom{n}{0} = 0$.
Moreover, since $w_j = j$, we can find a vector $e \in \BR^j$ such that $(w_j + \delta)^i -  (w_{j+1} + \delta)^i  = (w_j + \delta)^i -  (w_j  + \delta + 1)^i  = \sum_{\ell =1}^{j - 1} e_{\ell}\cdot (w_j + \delta)^{\ell}$. By plugging this identity into the series of equalities above yields
\begin{eqnarray*}
\sum_{j=1}^{n+1}(-1)^j \tbinom{n+1}{j} (w_j + \delta)^i &=& \sum_{j=1}^{n}(-1)^j  \left( \tbinom{n}{j} (w_j + \delta)^i -  (w_{j+1} + \delta)^i  \right) \\
& = & \sum_{j=1}^{n}(-1)^j  \tbinom{n}{j} \sum_{\ell =1}^{i - 1} e_{\ell} \cdot (w_j + \delta)^{\ell}\\
& = & \sum_{\ell =1}^{i - 1} e_{\ell}  \sum_{j=1}^{n}(-1)^j  \tbinom{n}{j} \cdot (w_j + \delta)^{\ell} = 0,
\end{eqnarray*}
where the last equality follows by induction. \qquad

\subsection{Proof of Lemma~\ref{lemma:technical-monotonic}.}
\label{sec:proof-technical-monotonic}
If $a_{j}$s are all nonnegative or nonpositive, then $\phi{f}(z)$ is monotone and has at most one root.
Otherwise, there is an $m$ such that $a_{j} \le 0$ when $j \le m$ and $a_{j} \ge 0$ when $j > m$. Denote $\phi_1(z) := -\sum_{j=1}^{m}a_j z^{i_j}$ and $\phi_2(z) = \sum_{j=m+1}^{M}a_j z^{i_j}$.
Obviously, $\phi(z) = \phi_2(z) - \phi_1(z)$. Suppose $x_0$ is a root of $\phi(z)$, that is $\phi_2(z_0) = \phi_1(z_0) \neq 0$. Given any $z_1 > z_0$, due to \eqref{coefficient-inequlity} we have that
\begin{align*}
\phi_2(z_1) &= \sum_{j=m+1}^{M}a_j \left(\frac{z_1}{z_0}\right)^{i_j} (z_0)^{i_j}  \ge \left(\frac{z_1}{z_0}\right)^{i_{m+1}} \phi_2(z_0) \\
\phi_1(z_1) &= \sum_{j=1}^{m}-a_j \left(\frac{z_1}{z_0}\right)^{i_j} (z_0)^{i_j}  \le \left(\frac{z_1}{z_0}\right)^{i_{m}} \phi_1(z_0).
\end{align*}
Combining these two inequalities yields
\begin{equation}\label{Inequality:monotone}
\phi(z_1) \ge \left(\frac{z_1}{z_0}\right)^{i_{m+1}} \phi_2(z_0)-\left(\frac{z_1}{z_0}\right)^{i_{m}} \phi_1(z_0)> \left(\frac{z_1}{z_0}\right)^{i_{m}} \left(\phi_2(z_0)-\phi_1(z_0)\right)=0.
\end{equation}
Similarly, for any $z_2 < z_0$, it holds that $\phi(z_2) < \left(\frac{z_2}{z_0}\right)^{i_{m}} \left(\phi_2(z_0)-\phi_1(z_0)\right)=0$.

Consequently, $z_0$ is the only root. Moreover, when $x_0$ is not a root and satisfies $\phi(z_0) \ge 0$, then according to \eqref{Inequality:monotone}
$\phi(z_1) > \left(\frac{z_1}{z_0}\right)^{i_{m}} \left(\phi_2(z_0)-\phi_1(z_0)\right) > \phi_2(z_0)-\phi_1(z_0) = \phi(z_0)$ implying that $\phi{f}(z)$ is monotonically increasing on $\{ z \; | \; \phi(z) \ge 0 \}$.




\end{document}